\documentclass[11pt]{amsart}

\usepackage{amssymb,latexsym,amsmath,extarrows, mathrsfs, amsthm,color,amsrefs}
\usepackage{graphicx}
\usepackage{MnSymbol}

\usepackage[margin=1in, centering]{geometry}
\usepackage[bookmarksnumbered, colorlinks, plainpages]{hyperref}
\usepackage{hyperref} 
\hypersetup{
    colorlinks=true,       
    linkcolor=blue,          
    citecolor=magenta,        
    filecolor=magenta,      
    urlcolor=cyan           
}

\usepackage{mathtools}
\mathtoolsset{showonlyrefs,showmanualtags}

\newtheorem{theorem}{Theorem}[section]

\newtheorem{lemma}[theorem]{Lemma}

\newtheorem{remark}[theorem]{Remark}

\newcommand{\ta}{\theta}

\newcommand{\Z}{\mathbb Z}
\newcommand{\R}{\mathbb R}
\newcommand{\C}{\mathbb C}
\newcommand{\T}{\mathbb T}
\newcommand{\N}{\mathbb N}

\newcommand{\kreis}{\mathbb{S}}

\def\mes{\mathrm{mes}}
\definecolor{deepgreen}{cmyk}{1,0,1,0.5}

\def\calC{\mathcal{C}}

\def\tor{\mathbb{T}}

\def\Mat{\mathrm{Mat}}

\def\diam{\mathrm{diam}}

\title[H\"older continuity of IDS]{H\"older continuity of the integrated density of states for quasi-periodic Jacobi block matrices}

\author[R.\ Han]{Rui Han}
\address{Department of Mathematics \\ Louisiana State University  \\  Baton Rouge, LA 70803, USA}
\email{rhan@lsu.edu}

\author[W.\ Schlag]{Wilhelm Schlag}
\address{Department of Mathematics \\ Yale University \\ New Haven, CT 06511, USA}
\email{wilhelm.schlag@yale.edu}

\thanks{
The authors gratefully acknowledge  partial support by the NSF; R.\ Han through DMS-2143369, and 
W.\ Schlag through DMS-2350356.}

\begin{document}

\begin{abstract}
In this paper, we prove H\"older continuity of the integrated density of states for discrete quasiperiodic Jacobi $d\times d$ block matrices with Diophantine frequencies. 
The H\"older exponent is shown to be any $\beta$ such that $0<\beta<1/(2\kappa^d)$, where $\kappa^d$ is the acceleration, i.e.,  the slope of the sum of the top $d$ Lyapunov exponents in the imaginary direction of the phase.
This generalizes the H\"older continuity results in the Schr\"odinger operator setting in \cites{GS2,HS1}, and also strengthens them in that setting by covering more Diophantine frequencies.
The proof is built on a new scheme for obtaining a local zero count for finite-volume characteristic polynomials from a global one.
\end{abstract}

\maketitle

\section{Introduction}

In this paper we study Schr\"odinger cocycles on a strip with Hamiltonian
\begin{equation}\label{eq:BVsys}
   (H_{\omega,\theta} \Phi)_n   = B_{n+1}(\theta)\Phi_{n+1}+ B_n^{(*)}(\theta) \Phi_{n-1}+ V_n (\theta)\Phi_n
\end{equation}
where $F_n(\theta):=F(\theta+n\omega)$ for any $d\times d$-matrix valued function. We set $B^{(*)}(\theta)=(B(\theta))^*$ for $\theta\in \T$, and require it to be the analytic extension of $(B(\theta))^*$ off of the real torus. Here $\omega\in\T$ and we assume that $\omega$ is Diophantine, i.e.,
\begin{equation}\begin{aligned}\label{def:DC}
    \omega\in \mathrm{DC}:=&\bigcup_{a>0, A>1} \mathrm{DC}_{a,A}, \text{ where}\\
    \mathrm{DC}_{a,A}=&\left\{\omega\in \T:\, \|k\omega\|_{\T}\geq \frac{a}{|k|^A}\, \text{ for all } k\in \Z\setminus \{0\}\right\}.
\end{aligned}  
\end{equation}
We further assume that $B,V\in C^{\omega}(\T_{\delta}, \Mat(d,\C))$ are analytic, where $$\T_{\delta}:=\{\theta+i\varepsilon:\, \theta\in \T,\, \varepsilon\in \R,  \text{ and } |\varepsilon|\leq \delta\}$$with some positive $\delta>0$. 
We assume throughout the paper that $V$ is Hermitian, and that~$B$ is invertible ($\det B(\theta)\neq 0$ for any $\theta\in \T_{\delta}$). 
The difference equation $H_{\omega,\theta} \Phi=E\Phi$ is equivalent to the cocycle 
\begin{align}\label{def:M_E}
  \calC &\::  (\theta,\Psi)\in \T\times \C^{2d}\mapsto (\theta+\omega, M_E(\theta)\Psi), \notag  \\
  \quad M_{E}(\theta) &= \left [ \begin{matrix}
   ( E-V(\theta))B(\theta)^{-1}  & -B^{(*)}(\theta) \\
  B(\theta)^{-1} & 0 
\end{matrix}\right]  
\end{align}
in the sense that for $n\ge1$, 
\begin{equation}\label{eq:MnE}
\calC^n (\theta,\Psi)=(\theta+n\omega, M_{n,E}(\theta)\Psi), \quad M_{n,E}(\theta)=\prod_{j=n-1}^0 M_E(\theta+j\omega), \quad 
\Psi_n:=\binom{B_{n}\Phi_n}{\Phi_{n-1}}
\end{equation} 
satisfies $\Psi_n = M_{n,E} (\theta) \Psi_0$.  The model~\eqref{eq:BVsys} was studied systematically by S.~Klein~\cite{SK_block}, and Duarte and Klein~\cites{DK1, DK2} within a very general framework, including the identically singular case when $\det B(\theta)\equiv 0$ on $\T^k$, $k\geq 1$. These models generalize the quasi-periodic Schr\"odinger Hamiltonians on a strip considered by Bourgain and Jitomirskaya~\cite{BJ}, who proved localization perturbatively for large disorder. 

Since $M_E(\theta)$ is (complex) symplectic, see \eqref{def:symp}, for $\theta\in \T$, the Lyapunov exponents $\{L_j(\omega,M_E)\}_{j=1}^{2d}$, see definition in \eqref{def:LE}, satisfy $L_{2d+1-j}=-L_j$ for $1\le j\le d$. 
Let $L^d(\omega,M_E)$ be the sum of the top $d$ Lyapunov exponents, and let $\kappa^d(\omega,M_E)$ be the acceleration, defined to be the right-derivative of $L^d(\omega,M_E(\cdot+i\varepsilon))$ in the imaginary direction $\varepsilon$, see \eqref{eq:acc}.

Throughout the paper we shall fix an energy $E_0\in \R$ such that $L_d(\omega,E_0)\geq \tau>0$. It is known, see \cite{AJS}, that $L^d(\omega,M_E(\cdot+i\varepsilon))$ and $L_d(\omega,M_E)$ are continuous in $E,\varepsilon$ and $\kappa^d(\omega,M_E)$ is upper semi-continuous in $E,\varepsilon$, hence we can assume that in a neighborhood $I_{E_0}$ of $E_0$, the following holds uniformly in $E\in I_{E_0}$ and $|\varepsilon|\leq \delta$, $\varepsilon\in \R$, for some $\delta=\delta(E_0)>0$:
\begin{align}\label{def:IE0}
\begin{cases}
L_d(\omega,M_E(\cdot+i\varepsilon))\geq \tau/2,\\
\kappa^d(\omega,M_E(\cdot+i\varepsilon))\leq \kappa^d(\omega,M_{E_0})\\
L^d(\omega,M_E(\cdot+i\varepsilon))\leq L^d(\omega,M_E)+2\pi \kappa^d(\omega,M_{E_0})\varepsilon.
\end{cases}
\end{align}
Note the third inequality is a corollary of the second one.

The density of states $\mathcal{N}(\omega,\cdot)$ is the limiting cumulative distribution function of the finite-volume eigenvalues.  
Let $H_{\omega,\theta}|_{\Lambda}$ be the restriction of $H_{\omega,\theta}$ to the interval $\Lambda=[a,b]\subset \Z$.
Let $\{E_{\Lambda,j}(\omega,\theta)\}_{j=1}^{|\Lambda|}$ be the eigenvalues of $H_{\omega,\theta}|_{\Lambda}$.
Consider
\begin{align}
    N_{\Lambda}(\omega,E,\theta):=\frac{1}{|\Lambda|}\sum_{j=1}^{|\Lambda|}\chi_{(-\infty,E)}(E_{\Lambda,j}(\omega,\theta)).
\end{align}
It is well-known that the weak-limit
\begin{align}
    \lim_{a\to-\infty, b\to \infty}\mathrm{d} N_{\Lambda}(\omega,E,\theta)=:\mathrm{d} \mathcal{N}(\omega,E).
\end{align}
exists and is independent of $\theta$. In this paper we prove the following result.

\begin{theorem}\label{thm:main}
Let $E_0\in \R$ be a fixed energy for which $L_d(\omega,E_0)>0$.
For any $\omega\in \mathrm{DC}$, the integrated density of states of $H_{\omega,\theta}$ satisfies
    \begin{align}
        |\mathcal{N}(\omega,E)-\mathcal{N}(\omega,E')|\leq |E-E'|^{\beta},
    \end{align}
    for any $E,E'\in I_{E_0}$, a neighborhood of $E_0$, and any H\"older exponent $0<\beta<1/(2\kappa^d(\omega,E_0))$. 
\end{theorem}

\begin{remark}
    The exponent $2\kappa^d(\omega,E)$ comes from the local zero count of the finite-volume characteristic polynomials, which is derived from the global zero count. When $d\geq 2$, this is the zero count of $f_n(\theta,E)$, see \eqref{def:fn}. If the blocks $B,V$ satisfy additional symmetries, then it is possible to deduce  smaller local zero counts, and hence improved H\"older exponents. We refer interested readers to \cite[Theorem 1.4 \& Theorem 1.6]{HS3} where the effects of additional symmetries are studied. 
\end{remark}

As for the history of this type of regularity result: $\log$--H\"older continuity was established by Craig and Simon in \cite{CS} for general ergodic potentials. The stronger H\"older continuity of the integrated density of states associated with~\eqref{eq:BVsys} was first proved by Goldstein and the second author in \cite{GS1} for the scalar case $d=1$.
Later, \cite{GS2}*{Theorem 1.1} established that for a trigonometric polynomial $f$ of degree $k_0\ge1$ and assuming $\omega\in \mathrm{DC}_{\mathrm{strong}}$, see \eqref{eq:SDC}, and 
$L(\omega,E)>0$  the IDS is $\beta$-H\"older continuous for any $\beta<1/(2k_0)$. 
This result, for the same set of frequencies, was improved to $\beta<1/(2\kappa^1(\omega,E))$ in \cite{HS1}. Note for trigonometric $f$ one has $\kappa^1(\omega,E)\leq k_0$.
Theorem~\ref{thm:main}, when $d=1$, is an improvement of \cite{HS1}*{Theorem 1.3} in the sense that it covers more general Diophantine frequencies: $\mathrm{DC}_{\mathrm{strong}}\subset \mathrm{DC}$. When $d\geq 2$, Theorem~\ref{thm:main} is completely new.
 
For the almost Mathieu operator and $\alpha\in \mathrm{DC}$, the IDS was proved to be $1/2$-H\"older if $|\lambda|\neq 0,1$, see Avila, Jitomirskaya~\cite{AJ}.
The $1/2$-H\"older exponent was also proved for Diophantine $\alpha$ and $f=\lambda\, g$ with $g$ being analytic and the coupling constant $|\lambda|$ being small, first in the perturbative regime (smallness depends on $\alpha, g$) by Amor~\cite{Amor}, and then in the non-perturbative regime \cite{AJ} (with dependence on $\alpha$ removed).
For quasi-periodic long-range operators with large trigonometric polynomial potentials and Diophantine frequencies, the H\"older exponent in \cite{GS2} was improved recently in the perturbative regime~\cite{GYZ}.
These results are proved using the reducibility method.

For the proof of Theorem~\ref{thm:main} we first single out the Schr\"odinger operator case, where $d=1$, $B\equiv 1$ and $V=v$ is a non-constant analytic function. Sections~\ref{sec:sch1}, \ref{sec:sch2} present this case in details. The argument is variant of the one in~\cite{GS2}, but is simpler and more robust.
The difference lies in how we derive a favorable local zero count for a finite volume determinant relative to some interval $\Lambda\subset\Z$. In~\cite{GS2}, this is accomplished by slightly varying the size of~$\Lambda$, but not the position. Here we keep the size fixed but change the position. The strip case ($d>1$ in~\eqref{eq:BVsys}) follows the same outline but is more involved because it relies on zero counting techniques that apply to the Jacobi block case. Thankfully,  these were developed in~\cite{HS3}, so we can treat them as a black box, cf.~Theorem~\ref{thm:acc=zeros}. 

We also point out that the Aubry dual of a subcritical Schr\"odinger operator $H$, with trigonometric potential with degree $d$, is a special family of $d\times d$ Jacobi block-valued matrix $\hat{H}$ with uniformly positive $L_d(\omega,E)>0$, see e.g.~\cite[Corollary 1.12]{HS2}. An easy computation shows that $\kappa^d(\omega,E)\leq d$, but one can show the local zero count is at most $2$, instead of $2d$. This is a small modification of Lemma \ref{lem:local_zero*} with considering not only $f_n(e^{2\pi ik\omega}z,E)$ with $|k|<(1-\varepsilon)n/2$, but also $f_n(e^{2\pi ik\omega/d}z,E)$ with $|k|<(1-\varepsilon)nd/2$. This is a feature due to the special form of $\hat{H}$, and has been explored in the proof of arithmetic Anderson localization of $\hat{H}$ in \cite[Theorem 1.7]{HS2} and \cite[Corollary 1.3]{HS3}. The local zero count being at most $2$ then implies almost $1/2$ H\"older continuity of the integrated density of states of the original subcritical Schr\"odinger with arbitrary trigonometric potentials with Diophantine frequencies.

Finally, we comment that the (non-)H\"older continuity of the integrated density of states for Liouville frequencies has also been studied \cite{YZ,HZ,ALSZ,HS4} when $d=1$, 
and it would be natural to combine those considerations with the techniques developed here for $d\geq 2$.

Our proofs are essentially self-contained. The only ingredients we require are the large deviation estimates, developed in \cites{BG,GS1,HS3} and global zero counts of finite-volume determinants in \cites{HS1,HS3}. 
As we mentioned, the Diophantine condition \eqref{def:DC} is weaker than that in \cites{GS2,HS1} where it was defined to be
\begin{align}\label{eq:SDC}
 \mathrm{DC}_{\mathrm{strong}}:= \bigcup_{a>1, c>0} \left\{\omega\in \T:\, \|k\omega\|_{\T}\geq \frac{c}{|k|(\log |k|)^a}\, \text{ for all } k\in \Z\setminus \{0,\pm 1\}\right\}.
\end{align}
The Anderson localization result of \cite{HS1} holds for the weaker Diophantine condition \eqref{def:DC}, see \cite{HS1}*{Remark 1.9}. However,  the H\"older continuity result of \cite{HS1} holds under \eqref{eq:SDC} due to the need for a sharp large deviation estimate as in \cite{GS2}. 
 
We organize the paper as follows: Section~\ref{sec:Pre} contains the preliminary results. Sections~\ref{sec:sch1} and \ref{sec:sch2} prove Theorem~\ref{thm:main} in the Schr\"odinger case, more specifically, Sec.~\ref{sec:sch1} presents the crucial local factorization of the finite volume characteristic polynomials which lays the foundation of the proof of Theorem~\ref{thm:main} in Sec.~\ref{sec:sch2}. The Jacobi block-valued case is proved in Sections~\ref{sec:block1} and \ref{sec:block2}. Finally, the proof of Lemma~\ref{lem:numerator_diag} is contained in Sec.~\ref{sec:numerator}.

\section{Lyapunov exponents, large deviations, and the Green's function}\label{sec:Pre}

Throughout, we adhere to the following notations. 
For a function $g$ on $\T$, we denote its $L^p(\T)$ norm by $\|g\|_{\T, p}$, and we write $\langle g\rangle:=\int_{\T } g(\theta)\, \mathrm{d}\theta$ for averages.
For $x\in \R$, let $\|x\|_{\tor}:=\mathrm{dist}(x,\Z)$ be the distance to the nearest integer vector. 
Let $\kreis^1:=\{z\in \C:\, |z|=1\}$ be the unit circle.
For a set $U\subset\R$, let $\mathrm{mes}(U)$ be its Lebesgue measure.
In various notations we shall replace $\theta+i\varepsilon$ with the complex variable $z$ when $z=e^{2\pi i(\theta+i\varepsilon)}$.

\subsection{Transfer matrices}
As we mentioned in the introduction, $M_E$ as in \eqref{def:M_E} is the transfer matrix associated to the block-valued operator $H_{\theta}$ in \eqref{eq:BVsys}.
A complex matrix $M\in \mathrm{Mat}(2d,\C)$ is symplectic if 
\begin{align}\label{def:symp}
    M^*\Omega M=\Omega,
\end{align}
where 
\begin{align}
    \Omega=\left(\begin{matrix} 0 & I_d\\ -I_d &0\end{matrix}\right).
\end{align}
One can easily verify that for $M_E$ as in \eqref{def:M_E}, and $E\in \R
$, $\theta\in \T$,
\begin{align}
    (M_E(\theta))^* \Omega M_E(\theta)=\Omega.
\end{align}
However for $\theta\in \T_{\delta}\setminus \T$, $M_E(\theta)$ is in general not symplectic.
$M_{n,E}$ is the $n$-step propagator, see~\eqref{eq:MnE}. 

\subsection{Lyapunov exponents}
Let $(\omega, A)\in (\T, C^{\omega}(\T, \mathrm{Mat}(k,\C)))$. 
Let 
\begin{align}
A_n(\omega,\theta)=A(\theta+(n-1)\omega)\cdots A(\theta).
\end{align}
Let the finite-scale and infinite-scale Lyapunov exponents be defined as 
\begin{align}\label{def:LE}
L_j(\omega, A,n):=\frac{1}{n}\int_{\T} \log \sigma_j(A_n(\omega,\theta))\, \mathrm{d}\theta, \text{ for } 1\leq j\leq k,
\end{align}
where $\sigma_j(A)$ is the $j$-th singular value of $A$, and the $j$-th Lyapunov exponent
\begin{align}
L_j(\omega, A)=\lim_{n\to\infty}L_j(\omega, A,n).
\end{align}
It is easy to see that for $1\leq j\leq k$,
\begin{align}
L^j(\omega, A,n):=\sum_{\ell=1}^j L_{\ell}(\omega, A,n)=\frac{1}{n}\int_{\T}\log \|\textstyle{\bigwedge^j} A_n(\omega,\theta)\|\, \mathrm{d}\theta,
\end{align}
where $\textstyle{\bigwedge^j} A$ is the $j$-th exterior power of $A$. Similarly $L^j(\omega,A)=\sum_{\ell=1}^j L_{\ell}(\omega, A)$.
We also denote the phase complexified Lyapunov exponents 
\begin{align}
\begin{cases}
    L_j(\omega,A(\cdot+i\varepsilon),n)=:L_{j,\varepsilon}(\omega,A,n),\\
    L_j(\omega,A(\cdot+i\varepsilon))=:L_{j,\varepsilon}(\omega,A),\\
    L^j(\omega,A(\cdot+i\varepsilon),n)=:L^j_{\varepsilon}(\omega,A,n),\\
    L^j(\omega,A(\cdot+i\varepsilon))=:L^j_{\varepsilon}(\omega,A),
\end{cases}
\end{align}
here $\varepsilon\in \R$.
Since $M_E(\theta)$ is symplectic for $\theta\in \T$, for each $1\leq j\leq d$ and $n\geq 1$,  we have
\begin{align}
    \sigma_j(M_{n,E}(\theta))=1/\sigma_{2d+1-j}(M_{n,E}(\theta)),
\end{align}
and hence
\begin{align}
    L_{j,\varepsilon=0}(\omega,M_E,n)=-L_{2d+1-j,\varepsilon=0}(\omega,M_E,n).
\end{align}
However the above is in general not true if $\varepsilon\neq 0$.

\subsection{Avila's acceleration}
Let $(\omega, A)\in (\T, C^{\omega}(\T, \mathrm{SL}(2, \R)))$.
The (top) Lyapunov exponent $L^1_{\varepsilon}(\omega,A)=L_{1,\varepsilon}(\omega, A)$ is a convex and even function in $\varepsilon$.
Avila defined the acceleration to be the right-derivative as follows:
\begin{align}\label{eq:acc}
\kappa^1_{\varepsilon}(\omega, A):=\lim_{\varepsilon'\to 0^+} \frac{L^1_{\varepsilon+\varepsilon'}(\omega, A)-L^1_{\varepsilon}(\omega, A)}{2\pi \varepsilon'}.
\end{align}
As a cornerstone of his global theory \cite{Global}, he showed that for $A\in \mathrm{SL}(2,\R)$ and irrational $\alpha$, $\kappa_{\varepsilon}^1(\omega, A)\in \Z$ is always quantized. 
The concept of acceleration was further extended to  $(\omega,A)\in (\T,C^{\omega}(\T, \mathrm{Mat}(k,\C)))$ by Avila, Jitomirskaya, and Sadel~\cite{AJS}, where for $1\leq j\leq k$,
\begin{align}
    \kappa^j_{\varepsilon}(\omega,A):=\lim_{\varepsilon'\to 0^+} \frac{L^j_{\varepsilon+\varepsilon'}(\omega, A)-L^j_{\varepsilon}(\omega, A)}{2\pi \varepsilon'}.
\end{align}
For the remainder of the paper, when $\varepsilon=0$, we shall omit $\varepsilon$ from various notations involving Lyapunov exponents and accelerations.
On some occasions, we shall also omit $\omega$ and $M_E$ in $L^j(\omega,M_E)$, $L_j(\omega,M_E)$ and $\kappa^d(\omega,M_E)$.
We will require the following tools. 
Note that we do not distinguish the various $\gamma$'s in the following Lemmas~\ref{lem:Ln-L}, \ref{lem:upperbd}, \ref{lem:upperbd_E'}, \ref{lem:LDTsig}, \ref{lem:un_LDT}, \ref{lem:numerator_diag}, \ref{lem:deno}, and Theorems~\ref{thm:acc=zeros}, \ref{thm:acc=zeros*}.

\subsection{Large deviation estimates for monodromy matrices}

The following rate of convergence holds for Lyapunov exponents, see~\cite{GS1}*{Lemma 10.1}. 

\begin{lemma}\label{lem:Ln-L}
Let $\omega\in \mathrm{DC}$.
Suppose $L_d(\omega,M_E)\geq \nu>0$, then there exists $\gamma>0$ such that for any $|\varepsilon|\leq \delta$ and $n\geq n(\nu)$, we have
\begin{align}
    L^d_{\varepsilon}(\omega,M_E)\leq L^d_{\varepsilon}(\omega,M_E,n)\leq L^d_{\varepsilon}(\omega,M_E)+n^{-\gamma}.
\end{align}
\end{lemma}
The proof of such effective convergence rate uses the Avalanche Principle, which requires the gap condition $L_d(\omega,M_E)>0$, see \cites{GS1, DK1}. 

Lemma~\ref{lem:upperbd} states a uniform upper bound on monodromy matrices. These are quite standard, see for example~\cite{GS2}*{Section 4}. We have no need for the sharper upper bound of~\cite{GS2}*{Proposition~4.3}, which requires the stronger Diophantine condition~\eqref{eq:SDC}. Instead, the following estimate holds. 

\begin{lemma}\label{lem:upperbd}
For $\omega\in \mathrm{DC}$,
there exists $\gamma>0$ so that for all large $n$, one has 
\begin{align}
\frac{1}{n} \log \big \|\textstyle{\bigwedge^d} M_{n,E}(\theta+i\varepsilon)\big\|\leq {L}^{d}_{\varepsilon}(\omega,M_E)+ n^{-\gamma},
\end{align}
uniformly in $\theta\in \tor$ and $|\varepsilon|\leq \delta$, $\varepsilon\in\R$. 
\end{lemma}
Note that we replaced $L_{\varepsilon}^d(\omega,M_E,n)$ on the right-hand-side of Lemma~\ref{lem:upperbd} with the true Lyapunov exponent  by means of Lemma \ref{lem:Ln-L}.
Strictly speaking, the exterior product does not appear in~\cite{GS2}. But due to the multiplicativity of ${\bigwedge}^j$, the same techniques apply here. 
The uniform upper bound is stable in a neighborhood of $E$, in fact one has the following:
\begin{lemma}\label{lem:upperbd_E'}
Under the same conditions as Lemma \ref{lem:upperbd}, for large $n$, for any $E'\in \C$ such that $|E'-E|<n^{-2}e^{-n^{\gamma}}$,
\begin{align}
\frac{1}{n}\log \big \|\textstyle{\bigwedge^d} M_{n,E'}(\theta+i\varepsilon)\big\|\leq {L}_\varepsilon^{d}(\omega,M_E)+ 2n^{-\gamma},
\end{align}
uniformly in $\theta\in \T$ and $|\varepsilon|\leq \delta$, $\varepsilon\in \R$.
\end{lemma}
\begin{proof}
One can prove this lemma inductively based on the following:
\begin{lemma}\label{lem:hn}
        Suppose $\{h_n\}_{n=1}^\infty$ and $\{\tilde{h}_n\}_{n=1}^{\infty}$ are two sequences of nonnegative numbers. Suppose that 
        \begin{align*}
            \tilde{h}_1&\leq h_1+x \\
            \tilde{h}_n&\leq h_n+x\sum_{k=1}^{n-1} \tilde{h}_k h_{n-k-1}.
        \end{align*}
        Assume further that for some constants $C\ge1$, $L>0$ and $0<b<1$, 
        \[
        h_n\le C \exp( nL + n^b)\qquad \forall\; n\ge1
        \]
        Then for any $n\ge1$,
        \[
        \tilde{h}_n\le 2C \exp( nL + n^b)
        \]
        for all $0\le x\le (2Cn)^{-1}e^{-n^b}$. 
    \end{lemma}
\begin{proof}
We give a quick proof of Lemma \ref{lem:hn} by induction in $n$. By inspection, the claim holds for $n=1$. Let $n\ge2$ and assume that it holds up to $n-1$. Then by the recursion, 
    \begin{align*}
        \tilde{h}_n &\leq  C \exp( nL + n^b)  + 2C^2 x\sum_{k=1}^{n-1} \exp( kL + k^b) \exp( (n-k-1)L + (n-k-1)^b) \\
        &\leq  C \exp( nL + n^b)  \big[1+  2Cxn e^{n^b}\big]\leq 2C \exp( nL + n^b) 
    \end{align*}
    provided $0\le x\le (2Cn)^{-1}e^{-n^b}$.
    \end{proof}
    Coming back to the proof of Lemma \ref{lem:upperbd_E'}, we apply Lemma \ref{lem:hn} to $h_n=\sup_{\theta}\|{\bigwedge}^j M_{n,E}(\theta)\|$ and $\tilde{h}_n=\sup_{\theta}\|{\bigwedge}^j M_{n,E'}(\theta)\|$ where $x=|E'-E|$.
\end{proof}

For a general $j\in \{1,...,2d\}\setminus \{d\}$, the following upper bound with error $o(1$) is sufficient. We do not have access to the sharper estimate for the top Lyapunov exponent $L^d$ as in Lemma~\ref{lem:Ln-L} because of the possible absence of a gap condition.
\begin{lemma}\label{lem:upper_j_neq_d}
    For $\omega\in \mathrm{DC}$, for any $E'\in \C$ such that $|E'-E|<o(1)$, for all large $n$, one has
    \begin{align}
        \frac{1}{n} \log \big \|\textstyle{\bigwedge^j} M_{n,E'}(\theta+i\varepsilon)\big\|\leq {L}^{j}_{\varepsilon}(\omega,M_E)+o(1),
    \end{align}
    uniformly in $\theta\in\T$ and $|\varepsilon|\leq \delta$, $\varepsilon\in \R$.
\end{lemma}

The following large deviation estimates play a crucial role in our argument.  These results were first established in Lemma~1.1 of~\cite{BG} by Bourgain and Goldstein, and further developed in~\cites{GS1,GS2}.

\begin{lemma}\label{lem:LDTsig}
For $\omega\in \mathrm{DC}$, there exists $\gamma>0$ such that for any $|\varepsilon|\leq \delta$, $\varepsilon\in \R$, and $n$ large enough, the following large deviation set 
\begin{align}\label{def:badnE}
\mathcal{B}_{n,E,\varepsilon}:=\left\{\theta\in\T:\, \frac{1}{n}\log \|\textstyle{\bigwedge^d}  M_{n,E}(\theta+i\varepsilon)\|\leq {L}_{\varepsilon}^d(\omega,M_E,n)-n^{-\gamma}\right\}
\end{align}
satisfies $\mathrm{mes}(\mathcal{B}_{n,E,\varepsilon})\leq e^{-n^\gamma}$. 
\end{lemma}

We will also use the Lipschitz continuity of $L^{d}_{\varepsilon}(\omega,M_E,n)$ with respect
to $\varepsilon$.
\begin{lemma}\label{lem:Lip_eps}\cite{GS2}*{Corollary 4.2}
There exists $C=C(B,V,|E|)>0$, such that for each $1\leq j\leq 2d$,
    \begin{align}\label{eq:Lip_eps}
        |L^j_{\varepsilon}(\omega,M_E,n)-L^j_{\varepsilon'}(\omega,M_E,n)|\leq C  |\varepsilon-\varepsilon'|,
    \end{align}
   for all sufficiently small $|\varepsilon|$, and  
    uniformly in $n$. 
\end{lemma}

\subsection{Green's function and Poisson formula}
\subsubsection{Schr\"odinger case}
In the Sch\"odinger case, the resolvent can be expressed in terms of Dirichlet determinants. 
In fact, let
\begin{align}\label{def:Dn}
    D_n(\theta,E):=\det(H_{\theta}|_{[0,n-1]}-E),
\end{align}
where $H_{\theta}|_{[0,n-1]}$ is the restriction of $H_{\theta}$ to the interval $[0,n-1]$ with Dirichlet boundary condition.
For $[\ell,m]\subset\Z$, let $G_{[\ell,m]}^E(\theta):=(H_{\theta}|_{[\ell,m]}-E)^{-1}$ be the Green's function.
The following Possion formula is well-known: for $\ell\leq k<j\leq m$, the $(k,j)$-th entry of the Green's function satisfies
\begin{align}\label{eq:Poisson_S}
    G_{[\ell,m]}^E(\theta;k,j)=\frac{D_{k-\ell}(\theta+\ell\omega,E)\cdot D_{m-j}(\theta+(j+1)\omega,E)}{D_{m-\ell+1}(\theta+\ell\omega,E)}.
\end{align}
The following large deviation estimate for $D_n$ was established in \cite{GS2}. 
\begin{lemma}\label{lem:un_LDT}\cite{GS2}*{Proposition 2.11}
If $L^1(\omega,E)>\tau>0$, there exists a positive constant $\gamma>0$ such that for $n$ large enough,
\begin{align*}
\mes\{\theta\in \T:\, n^{-1} \log |D_n(\theta+i\varepsilon,E)|<L^1_{\varepsilon}(\omega,E,n)-n^{-\gamma}\}<e^{-n^{\gamma}}.
\end{align*}
\end{lemma}

\subsubsection{Jacobi block case}
In the Jacobi-block case, we work with finite volume Hamiltonians under periodic boundary conditions. As in \cite{HS3}, for $\ell,m\in \Z$ and $\ell<m$, we define 
\begin{align}\label{def:H_per}
&H^p_{\theta}|_{[\ell d,md-1]}\\
=
&\left(\begin{matrix}
V(\theta+(m-1)\omega) & B^{(*)}(\theta+(m-1)\omega) & &  &B(\theta+\ell\omega)\\
B(\theta+(m-1)\omega) &V(\theta+(m-2)\omega) &\ddots \\
& \ddots &\ddots &\ddots \\
& &\ddots &\ddots &B^{(*)}(\theta+(\ell+1)\omega)\\
B^{(*)}(\theta+\ell \omega) & & &B(\theta+(\ell+1)\omega) &V(\theta+\ell\omega)
\end{matrix}\right),
\end{align}
Let 
\begin{align}\label{def:fn}
f_n(\theta,E):=\det(H^p_{\theta}|_{[0,nd-1]}-E).
\end{align}
For any $[\ell d,m d-1]\subset \Z$ with $\ell,m\in \Z$, $\ell<m$,
\begin{align}\label{def:Green}
G_{[\ell d,md-1]}^{p,E}(\theta):=(H^p_{\theta}|_{[\ell d,md-1]}-E)^{-1}
\end{align}
is the finite volume Green's function with periodic boundary conditions. 
For any $\ell d\leq k<j\leq md-1$, Cramer's rule yields the following analogue of~\eqref{eq:Poisson_S}
\begin{align}\label{eq:mufn}
G^E_{[\ell d,md-1]}(\theta;k,j)=\frac{\mu^E_{[\ell d,md-1],k,j}(\theta)}{f_{m-\ell}(\theta+\ell \omega,E)},
\end{align}
where $\mu^E_{[\ell d,md-1],k,j}(\theta)$ is the determinant of the submatrix of $(H_{\theta}^p|_{[\ell d, md-1]}-E)$ defined by deleting the $k$-th row and $j$-th column. 
The following upper bound for the numerator $\mu$ was proved in \cite{HS3}*{Lemma 2.6}.
\begin{lemma}\label{lem:numerator}
Let $\omega\in \mathrm{DC}$.
Let $3d\leq y\leq (n-1)d-1$ and $0\leq x\leq d-1$ or $(n-1)d\leq x\leq nd-1$. Set $\ell:=\lfloor y/d\rfloor$. Then 
for any $E\in \C$, $\varepsilon>0$, and uniformly in $\theta\in \T$, 
\begin{align}
|\mu^E_{[0,nd-1],x,y}(\theta)|\leq C_{d,B} \cdot e^{n(\langle\log |\det B|\rangle+\varepsilon)} \cdot \left(e^{\ell {L}^{d-1}(E)+(n-\ell){L}^d(E)}+e^{\ell {L}^d(E)+(n-\ell){L}^{d-1}(E)}\right),
\end{align}
where $L^j=L^j(\omega,M_E)$, provided $n>N(\varepsilon)$ is large enough. 
Here $C_{d,B}$ is a constant depending only on~$d$ and $\|B^{-1}\|_{\T, \infty}$. 
\end{lemma}

We will also need the following variant on the diagonal, which can be proved by similar arguments, see Section~\ref{sec:numerator}. 

\begin{lemma}\label{lem:numerator_diag}
    Let $\omega\in \mathrm{DC}$ and $x=y\in [2d,n-2d-1]$, 
    Then for $\gamma>0$ as in Lemma \ref{lem:upperbd}, $n$ large, and $E'\in \C$ such that $|E'-E|\leq n^{-2}e^{-n^{\gamma}}$, uniformly in $\theta\in \T$:
    \begin{align}
        |\mu^{E'}_{[0,nd-1],x,x}(\theta)|\leq  e^{n(L^d(E)+\langle \log |\det B|\rangle+5n^{-\gamma})}.
    \end{align}
\end{lemma}

Note the above estimates apply to $\mu_{[\ell d, md-1]}$ as well since
\begin{align}
    \mu^E_{[\ell d,md-1],k,j}(\theta)=\mu^E_{[0,(m-\ell)d-1],k-\ell d, j-\ell d}(\theta+\ell \omega).
\end{align}
The following connection between $f_n(\theta,E)$ and the transfer matrix $M_{n,E}(\theta)$ was proved in \cite{HS3}*{Lemma 2.7}.

\begin{lemma}\label{lem:detP}
For any $\theta\in \T_{\delta}$ and $E\in \C$,
\begin{align}
|f_n(\theta,E)|=  |\det (M_{n,E}(\theta)-I_{2d})| \cdot \prod_{j=0}^{n-1}|\det B(\theta+j\omega)| .
\end{align}
\end{lemma}

In \cite{HS3}*{Lemma 2.8} we established the following large deviation theorem for $f_n$.
\begin{lemma}\label{lem:deno}
Let $\omega\in \mathrm{DC}$, and $\gamma>0$ be as in Lemma \ref{lem:LDTsig}.
Assume $L_d(\omega,M_E)\geq \nu>0$.
There exist $\gamma>0$, $N_0>1$ large and $0<\kappa_0\ll 1$ so that the {\it $\kappa_0$-admissible} sequence 
\begin{align}\label{def:admissible}
\mathcal{N}:=\{n\geq N_0: \|n\omega\|_{\T}\leq \kappa_0\}
\end{align}
has the following property: 
for any $|\varepsilon|\leq \delta/2$, and all large $\kappa_0$-admissible $n$, the following large deviation set 
\begin{align}\label{def:B_fEn}
\mathcal{B}_{f,E,n,\varepsilon}:=
\big\{\theta\in \T: n^{-1}\log |f_{n}(\theta+i\varepsilon,E)|<\langle \log |\det B(\cdot+i\varepsilon)|\rangle +{L}^d_{\varepsilon}(\omega,M_E)- n^{-\gamma}\big\}
\end{align}
satisfies $\mathrm{mes}(\mathcal{B}_{f,E,n,\varepsilon})<e^{-n^{\gamma}}$.
\end{lemma}
\begin{remark}\label{rem:admissible}
    For every large integer $n>0$ there exists an admissible $\tilde n>0$ with $|n-\tilde n|\le C_*$ for some constant $C_*$.
\end{remark}

\subsection{Global zero counts}
\subsubsection{Schr\"odinger case}
The global zero count of $D_n$ was established in \cite{HS1}.
Let 
\begin{align}
    N_n(E,\varepsilon):=\#\{z\in \mathcal{A}_{\varepsilon}: D_n(z,E)=0\}.
\end{align}
\begin{theorem}\label{thm:acc=zeros}
Assume that for some $E\in \R$, $L^1(\omega,E)\geq \nu>0$.
Then for some $\gamma\in (0,1)$ and for $n$ large enough,
\begin{align}
\left| \frac{1}{2n} N_n (E,\delta/2)-\kappa^1(\omega,E)\right|\leq \delta^{-1} n^{-\gamma}.
\end{align}
\end{theorem}

\subsubsection{Jacobi block case}
The global zero count of $f_n$ was established in \cite{HS3}.
Let
\begin{align}
    N_n^p(E,\varepsilon):=\#\{z\in \mathcal{A}_{\varepsilon}: f_n(z,E)=0\}.
\end{align}

The following result is the analogue of Theorem~\ref{thm:acc=zeros}, but its proof is substantially more involved due to the Grassmannian algebra entering the analysis.

\begin{theorem}\label{thm:acc=zeros*}\cite{HS3}*{Theorem 5.8}
Assume that for some $E\in \R$, $L_d(\omega,E)\geq \nu>0$.
Then for some $\gamma\in (0,1)$ and for $n$ large enough and $\kappa_0$-admissible,
\begin{align}
\left| \frac{1}{2n} N^p_n (E,\delta/2)-\kappa^d(\omega,E)\right|\leq \delta^{-1} n^{-\gamma}.
\end{align}
\end{theorem}

\section{Analysis of local zeros for the Schr\"odinger case}
\label{sec:sch1}

Heuristically speaking, the following Lemma \ref{lem:local_zero}  plays the role of the crucial Theorem~1.4 of \cite{GS2}. That result asserts a local zero count for the Dirichlet determinant on an interval 
$[1,n]$ with perturbed edges $[s^-, n+s^+]$ with $|s^{\pm}|\simeq \exp((\log n)^{\delta})$, $\delta>0$ small, while the lemma below is for intervals with fixed length $n$ but with shifted centers: $[k+1,k+n]$, $|k|<(1-o(1))n/2$.
Recall that in the Schr\"odinger case, $D_n(z,E)$ is the Dirichlet determinant, see \eqref{def:Dn}. { The parameter $\epsilon$ in the lemma is different from the one appearing in the complexification of the phase. But this should not cause any confusion.}

Recall that $E_0\in \R$ is a fixed energy for which $L(\omega,E_0)>0$.
Throughout Sections~\ref{sec:sch1} and \ref{sec:sch2}, we write for simplicity that $\kappa=\kappa^1(\omega,E_0)$. The energy $E$ will be assumed to be in a neighborhood $I_{E_0}$ of $E_0$, see \eqref{def:IE0}.

\begin{lemma}\label{lem:local_zero}
For each ball $B(z_0,r)$, with $z_0\in \kreis^1$ and $r\simeq r_n:=e^{-(\log n)^{C_0}}$ with $C_0>1$, there exists $|k|<(1-\epsilon)n/2$ such that $D_n(e^{2\pi ik\omega} z,E)$ has no more than $2\kappa$ zeros in $B(z_0,r)$. 
\end{lemma}
\begin{proof}
    Assume otherwise that $D_n(e^{2\pi ik\omega}z,E)$ has at least $2\kappa+1$ zeros in $B(z_0,r)$ for each $|k|<(1-\epsilon)n/2$.
    This implies for each such $k$ that $D_n(z,E)$ has at least $2\kappa+1$ zeros in $B(e^{-2\pi ik\omega}z_0,r)$.
    These balls are disjoint from each other due to the Diophantine condition on $\omega$: 
    \[\|(k_1-k_2)\omega\|\gtrsim \frac{1}{n^A}\gg e^{-(\log n)^{C_0}}.\]
    This implies $D_n(z,E)$ has at least $(2\kappa+1)(1-\epsilon)n>2\kappa (1+o(1)) n$ zeros in the annulus, which is a contradiction to Theorem \ref{thm:acc=zeros} and  $\kappa^1(\omega,E)\leq \kappa$, see \eqref{def:IE0}.
\end{proof}

\begin{lemma}\label{lem:un_lower_zeros}
For each $z_0\in \kreis^1$, let $k_0=k_0(z_0)$, $|k_0|<(1-\epsilon)n/2$ be as in Lemma \ref{lem:local_zero}.  
Then there exists an integer $C_0\in [1,2\kappa+1]$ such that $D_n(e^{2\pi i k_0\omega}z,E)$ has no zero in $B(z_0,4(C_0+1)r_n)\setminus B(z_0,4C_0r_n)$, and has at most $2\kappa$ zeros, denoted by $\{z_{n,1},...,z_{n,j_0}\}$, in $B(z_0,4C_0r_n)$. Furthermore, for $z\in B(z_0,(4C_0+3)r_n)$, we have the following lower bound:
\begin{align}\label{eq:un_lower_zeros}
    \log |D_n(e^{2\pi ik_0\omega}z,E)|\geq n(L(E,n)-n^{-\gamma/2})+\sum_{j=1}^{j_0}\log|z-z_{n,j}|.
\end{align}
\end{lemma}
\begin{proof}
    The existence of $C_0$ is evident from Lemma \ref{lem:local_zero}. 
    Let 
    \begin{align}
        Y_n(z):=\frac{D_n(e^{2\pi ik_0\omega}z,E)}{\prod_{j=1}^{j_0} (z-z_{n,j})}.
    \end{align}
    Since $Y_n(z)$ has no zeros in $B(z_0, 4(C_0+1)r_n)$,  $y_n(z):=n^{-1}\log |Y_n(z)|$ is a harmonic function in $B(z_0,4(C_0+1)r_n)$.
    We first derive the upper bound for $y_n(z)$ on $\partial B(z_0,4(C_0+1)r_n)$. 
    If $z\in \partial B(z_0,4(C_0+1)r_n)$, then for each $1\leq j\leq j_0$, $|z-z_{n,j}|\geq 4r_n$, which implies that
    \begin{align}
        \prod_{j=1}^{j_0} |z-z_{n,j}|\geq (4r_n)^{j_0}\geq (4r_n)^{2\kappa}\geq e^{-3\kappa(\log n)^{C_0}}.
    \end{align}
    Moreover,  by Lemma \ref{lem:upperbd_E'} the following upper bound holds
    \[n^{-1}\log|D_n(e^{2\pi ik_0\omega}z,E)|\leq \sup_{|\varepsilon|\leq 8C_0r_n} L_\varepsilon(E,n)+2n^{-\gamma}\leq L(E,n)+3n^{-\gamma}\]
    Therefore, for any $z\in \partial B(z_0,4(C_0+1)r_n)$, and hence for all $z\in B(z_0,4(C_0+1)r_n)$ by the maximum principle, that
    \begin{align}
        y_n(z)\leq L(E,n)+3n^{-\gamma}+3\kappa \frac{(\log n)^{C_0}}{n}<L(E,n)+4n^{-\gamma}.
    \end{align}
    By the preceding  $\tilde{y}_n(z):=L(E,n)+4n^{-\gamma}-y_n(z)$ is a non-negative harmonic function on $B(z_0, 4(C_0+1)r_n)$. By Harnack's inequality,
    \begin{align}
        \sup_{z\in B(z_0,(4C_0+3)r_n)} \tilde{y}_n(z)\leq C \inf_{z\in B(z_0,(4C_0+3)r_n)}\tilde{y}_n(z).
    \end{align}
    The large deviation estimate in Lemma \ref{lem:un_LDT} applied to $D_n(e^{2\pi i \theta},E)$ yields
    \begin{align}
        \mathrm{mes}\{\theta: n^{-1}\log |D_n(e^{2\pi i\theta},E)|<L(E,n)-n^{-\gamma}\}\leq e^{-n^{\gamma}}\ll r_n\simeq \diam(B(z_0,r_n)\cap \kreis^1).
    \end{align}
    Hence, there exists $z_1\in B(z_0,r_n)$ such that
    \begin{align}
        y_n(z_1)>L(E,n)-n^{-\gamma},
    \end{align}
    and thus
    \begin{align}
        \inf_{z\in B(z_0,(4C_0+3)r_n)}\tilde{y}_n(z)\leq \tilde{y}_n(z_1)\leq 5n^{-\gamma}.
    \end{align}
    This implies that for all $z\in B(z_0,(4C_0+3)r_n)$,
    \begin{align}
        y_n(z)\geq L(E,n)-Cn^{-\gamma}>L(E,n)-n^{-\gamma/2}, 
    \end{align}
    as desired. \end{proof}
    
Next, we pass to complexified energies. 
Let $\gamma_1$ be a constant close to $1$ such that
\begin{align}\label{def:gamma1}
    1>\gamma_1>\max((1+\gamma)/2, 1-\gamma/2)>\gamma.
\end{align}

\begin{lemma}\label{lem:un_lower_zeros_eta}
    Let $z_0,k_0,C_0$ be as in Lemma \ref{lem:un_lower_zeros}. Let $\eta\in \R$ be such that $|\eta|\leq Ce^{-n^{\gamma_1}}$ for some constant $C>0$, with $\gamma_1$ as in \eqref{def:gamma1}.
    Then $D_n(e^{2\pi ik_0\omega}z,E+i\eta)$ has the same  number of zeros as $D_n(e^{2\pi ik_0\omega}z,E)$ in $B(z_0,(4C_0+2\pm 1)r_n)$,  the zeros being denoted by $\{z_{n,1}^{\eta},...,z_{n,j_0}^{\eta}\}$. Here the count $j_0$ is the same for both balls~$B(z_0,(4C_0+2\pm 1)r_n)$. Furthermore, for any $z\in B(z_0,(4C_0+2)r_n)$, the following lower bound holds:
    \begin{align}\label{eq:un_lower_zeros_eta}
    \log |D_n(e^{2\pi ik_0\omega}z,E+i\eta)|\geq n(L(E,n)-n^{-\gamma/2})+\sum_{j=1}^{j_0}\log|z-z_{n,j}^{\eta}|.
\end{align}
\end{lemma}
\begin{proof}
The  zero count will be obtained via Rouch\'e's theorem.
By a standard telescoping argument, and with $z=e^{2\pi i(\theta+i\varepsilon)}$,
\begin{align}
    &|D_n(e^{2\pi ik_0\omega}z,E+i\eta)-D_n(e^{2\pi ik_0\omega}z,E)|\\
    \leq &\|M_{n,E+i\eta}(\theta+k_0\omega+i\varepsilon)-M_{n,E}(\theta+k_0\omega+i\varepsilon)\|\\
    \leq &\sum_{\ell=1}^n \|M_{n-\ell,E+i\eta}(\theta+(k_0+\ell+1)\omega+i\varepsilon)\|\cdot \|M_{E+i\eta}(\theta+(k_0+\ell)\omega+i\varepsilon)-M_E(\theta+(k_0+\ell)\omega+i\varepsilon)\|\\
    &\qquad \cdot \|M_{\ell-1,E}(\theta+k_0\omega+i\varepsilon)\|,
\end{align}
where by convention $M_{0,E}\equiv I_2$ is the identity matrix.
By the  upper bounds in Lemma \ref{lem:upperbd_E'}, note that $|\varepsilon|\lesssim r_n$, we have
\begin{align}
    \sup_{\theta\in \T} \sup_{|\eta'|\leq \eta} \|M_{\ell,E+i\eta'}(\theta+i\varepsilon)\|
    \leq &C\exp(\ell(L(E)+\ell^{-\gamma})).
\end{align}
One also has
\begin{align}
    \|M_{E+i\eta}(\theta+(k_0+\ell)\omega+i\varepsilon)-M_E(\theta+(k_0+\ell)\omega+i\varepsilon)\|\leq \eta.
\end{align}
Hence
\begin{align}
   &|D_n(e^{2\pi ik_0\omega}z,E+i\eta)-D_n(e^{2\pi ik_0\omega}z,E)|\\
    &\qquad \leq C\eta \cdot \exp(nL(E)+4n^{1-\gamma}) \\
    &\qquad \leq C \exp(nL(E)+4n^{1-\gamma}-n^{\gamma_1})\\
    &\qquad\leq \exp(n(L(E)-n^{\gamma_1-1}/2)),
\end{align}
where we used $\gamma_1>1-\gamma$.
We also have by Lemma \ref{lem:un_lower_zeros} that for $z\in \partial B(z_0,(4C_0+2\pm 1)r_n)$ 
\begin{align}
    |D_n(e^{2\pi ik_0\omega}z,E)|
    \geq &\exp(n(L(E,n)-n^{-\gamma/2}))\cdot r_n^{2\kappa}\\
    \geq &\exp(n(L(E,n)-2n^{-\gamma/2}))\\
    \geq &\exp(n(L(E)-2n^{-\gamma/2})),
\end{align}
where we used $|z-z_{j,n}|\geq r_n$.
Therefore, since $\gamma_1>1-\gamma/2$, we have
\begin{align}
    |D_n(e^{2\pi ik_0\omega}z,E+i\eta)-D_n(e^{2\pi ik_0\omega}z,E)|<\frac{1}{2}|D_n(e^{2\pi ik_0\omega}z,E)|,
\end{align}
for any $z\in \partial B(z_0,(4C_0+2\pm 1)r_n)$. Hence by Rouch\'e's theorem, 
$D_n(e^{2\pi ik_0\omega}z,E+i\eta)$ has the same number of zeros as $D_n(e^{2\pi ik_0\omega}z,E)$ in the disk $B(z_0,(4C_0+2\pm 1)r_n)$.
The proof of \eqref{eq:un_lower_zeros_eta} follows the same lines as that of \eqref{eq:un_lower_zeros} in Lemma \ref{lem:un_lower_zeros}.
\end{proof}

\section{H\"older continuity of the Schr\"odinger case}
\label{sec:sch2}

Recall that $E\in I_{E_0}$.
Our goal is to analyze the following for an arbitrary fixed $\theta$ as $N\to\infty$:
\begin{align}\label{eq:goal}
    &d_{\omega,N}(\theta,E-\eta,E+\eta)=\frac{1}{N}\mathrm{tr}(P_{[E-\eta, E+\eta)}(H_{\omega,\theta}|_{[0,N-1]})),
\end{align}
in which $H_{\omega,\theta}|_{[0,N-1]}$ is $H_{\omega,\theta}$ restricted to the interval $[0,N-1]$ with Dirichlet boundary condition and $P_{[E-\eta,E+\eta)}$ stands for the spectral projection.
Bounding \eqref{eq:goal} through the trace of the resolvent yields
\begin{align}\label{eq:dN<G}
    d_{\omega,N}(\theta,E-\eta,E+\eta)
    &\leq \frac{2\eta}{N} \mathrm{Im}(\mathrm{tr}(G_{[0,N-1]}^{E+i\eta}(\theta)))\\
    &=\frac{2\eta}{N}\sum_{k=0}^{N-1} \mathrm{Im} (G_{[0,N-1]}^{E+i\eta}(\theta;k,k)),
\end{align}
in which $G^{E+i\eta}_{[0,N-1]}(\theta)=(H_{\omega,\theta}|_{[0,N-1]}-(E+i\eta))^{-1}$ is the Green's function, i.e., the kernel of the resolvent.
Let an auxiliary parameter $\tau$ be defined as
\begin{align}\label{eq:choose_tau}
    \tau=\eta^{\frac{1}{2\kappa}}.
\end{align}
For any $n\in \N$, let 
\begin{align}
    \mathcal{B}_{n,h}^{\eta}:=\{\theta\in \T: u_n(e^{2\pi i\theta},E+i\eta)<L(E,n)-h\cdot  n^{-\gamma}\}.
\end{align}
be the large deviation set for $u_n(z,E+i\eta)=n^{-1}\log |D_n(z,E+i\eta)|$. 
We choose $m\in \N$ such that
$e^{-(2m)^{\gamma_1}}\simeq \eta$.
The set $\mathcal{B}_{2m,2}^{\eta}$ satisfies a large deviation estimate
\begin{align}\label{eq:mes_B2m_2}
    \mes\;\mathcal{B}_{2m,2}^{\eta}\leq e^{-(2m)^{\gamma}}.
\end{align}
In fact, the large deviation estimate of Lemma \ref{lem:un_LDT} applies directly to $\mathcal{B}_{2m,1}^0$, implying that
\begin{align}\label{eq:mes_B2m_1}
    \mathrm{mes}\; \mathcal{B}_{2m,1}^0\leq e^{-(2m)^{\gamma}}.
\end{align}
If $\theta\in (\mathcal{B}_{2m,1}^0)^c$, $|D_{2m}(e^{2\pi i\theta},E)|>\exp(2m(L(E,2m)-(2m)^{-\gamma}))$. A telescoping argument,  as in the proof of Lemma \ref{lem:un_lower_zeros_eta}, using that $\gamma_1>1-\gamma/2$, yields
\begin{align}
    |D_{2m}(e^{2\pi i\theta},E+i\eta)|>\frac{1}{2}|D_{2m}(e^{2\pi i\theta},E)|>\frac{1}{2}\exp(2m(L(E,2m)-(2m)^{-\gamma})).
\end{align}
This implies $(\mathcal{B}_{2m,1}^0)^c\subset (\mathcal{B}_{2m,2}^{\eta})^c$,
and hence \eqref{eq:mes_B2m_2}, when combining with \eqref{eq:mes_B2m_1}.

Let $\{\xi_j \,e^{2\pi i\theta_j}\}_{j=1}^{J_0}$, $\xi_j>0$ and $\theta_j\in \T$, be the zeros of $D_{2m}(z,E+i\eta)$. 
Combining the large deviation estimates in \eqref{eq:mes_B2m_2}, the uniform upper bound in Lemma \ref{lem:upperbd_E'} with
Cartan's estimate implies that 
\begin{align}\label{eq:B2m_theta_pm_rm}
\mathcal{B}_{2m, m^{\gamma/2}}^{\eta}\subset \bigcup_{j=1}^{J_0}(\theta_j-\tilde{r}_m, \theta_j+\tilde{r}_m),
\end{align}
with $\tilde{r}_m:=e^{-cm^{\gamma/2}}$ for some constant $c>0$, and $J_0\lesssim m$.

Recall $\eta\simeq e^{-(2m)^{\gamma_1}}$, then applying Lemma \ref{lem:un_lower_zeros_eta} to $n=2m$ and $z_0=e^{2\pi i\theta_j}$, for each $1\leq j\leq J_0$, yields some $|k_j|<(1-\epsilon)m$, $C_j\in [1,2\kappa+1]$, and $\ell_j\in [1,2\kappa]$, such that $w_j(z):=D_{2m}(z e^{2\pi ik_j\omega},E+i\eta)$
has $\ell_j$ zeros in the ball $B(e^{2\pi i\theta_j}, (4C_j+1)r_m)$, $r_m\simeq \exp(-(\log m)^{C_0})\gg \tilde{r}_m$.
Denoting the zeros of $D_{2m}(z e^{2\pi ik_j\omega},E+i\eta)$ in the ball $B(e^{2\pi i\theta_j},(4C_j+1)r_m)$ by $\{\xi_{j,\ell}\, e^{2\pi i\theta_{j,\ell}}\}_{\ell=1}^{\ell_j}$ with $\xi_{j,\ell}>0$, $\theta_{j,\ell}\in \T$, we have
\begin{align}\label{eq:wj_lower_zeros}
    \log |w_j(z)|\geq 2m(L(E,2m)-(2m)^{-\gamma/2})+\sum_{\ell=1}^{\ell_j}\log |z-\xi_{j,\ell}\,e^{2\pi i\theta_{j,\ell}}|,
\end{align}
for all $z\in B(e^{2\pi i\theta_j},(4C_j+2)r_m)$.
For each $1\leq j\leq J_0$ and $1\leq \ell\leq \ell_j$, we denote\begin{equation}
    \label{eq:Ijdef}
\begin{aligned}
    \tilde{I}_j:=&(\theta_j-(4C_j+2)r_m, \theta_j+(4C_j+2)r_m), \text{ and } \\
    I_{j,\ell}:=&(\theta_{j,\ell}-\tau,\theta_{j,\ell}+\tau).
\end{aligned}
\end{equation}
Recall $\tau$ is as in \eqref{eq:choose_tau} and $\tau\ll r_m$, hence for each $j,\ell$ we have $I_{j,\ell}\subset \tilde{I}_j$. 
Let $\tilde{\mathcal{B}}_{2m}^{\eta}:=\bigcup_{j=1}^{J_0} \tilde{I}_j$. 
Since $\tilde{r}_m\ll r_m$, by \eqref{eq:B2m_theta_pm_rm}, we clearly have $\mathcal{B}_{2m,m^{\gamma/2}}^{\eta}\subset \tilde{\mathcal{B}}_{2m}^{\eta}$.
We divide the analysis of $G_{[0,N-1]}^{E+i\eta}(\theta;k,k)$ into three cases.

\smallskip

{\bf Case 1.\ }  $\theta+k\omega-m\omega\in (\tilde{\mathcal{B}}_{2m}^{\eta})^c$.

\medskip

This is the easier, non-resonant case. Since $\theta+k\omega-m\omega\in (\tilde{\mathcal{B}}_{2m}^{\eta})^c\subset (\mathcal{B}^{\eta}_{2m,m^{\gamma/2}})^c$, 
we have
\begin{align}\label{eq:D2m_lower_non-resonant}
    |D_{2m}(\theta+k\omega-m\omega,E+i\eta)|
    \geq &\exp(2m(L(E,2m)-m^{\gamma/2}\cdot (2m)^{-\gamma}))\\
    \geq &\exp(2m(L(E,2m)-m^{-\gamma/2})).
\end{align}
Let $[a_k,b_k]:=[k-m,k+m-1]$ and define
\begin{align}
    \tilde{H}_{\theta}|_{[0,N-1],k}:=H_{\theta}|_{[a_k,b_k]} \oplus H_{\theta}|_{[0,N-1]\setminus [a_k,b_k]},
\end{align}
and $\tilde{G}^{E+i\ta}(\theta)=(\tilde{H}_{\theta}-(E+i\eta))^{-1}$ be the Green's function.
By the resolvent identity, we have
\begin{align}\label{eq:G=tG_case1}
    &G_{[0,N-1]}^{E+i\eta}(\theta;k,k)
    =\tilde{G}_{[0,N-1],k}^{E+i\eta}(\theta;k,k)\\
    &\qquad\qquad\qquad\qquad-G^{E+i\eta}_{[0,N-1]}(\theta;k,a_k-1)\cdot \tilde{G}^{E+i\eta}_{[a_k,b_k]}(\theta;a_k,k)\\
    &\qquad\qquad\qquad\qquad-G^{E+i\eta}_{[0,N-1]}(\theta;k,b_k+1)\cdot \tilde{G}^{E+i\eta}_{[a_k,b_k]}(\theta;b_k,k).
\end{align}
We bound two product terms above as follows:
\begin{align}\label{eq:G=tG_1_case1}
    \max(|G^{E+i\eta}_{[0,N-1]}(\theta;k,a_k-1)|,\, |G^{E+i\eta}_{[0,N-1]}(\theta;k,b_k+1)|)\leq \eta^{-1}.
\end{align}
For the other two product terms, by Cramer's rule, and using \eqref{eq:D2m_lower_non-resonant} for denominator lower bound and Lemma \ref{lem:upperbd_E'} for numerator upper bound, we have
\begin{align}\label{eq:G=tG_2_case1}
    |\tilde{G}^{E+i\eta}_{[a_k,b_k]}(\theta;a_k,k,E+i\eta)|&=\frac{|D_{b_k-k}(\theta+(k+1)\omega,E+i\eta)|}{|D_{2m}(\theta+a_k\omega,E+i\eta)|}\\
    &\qquad \leq e^{-m(L(E)-Cm^{-\gamma/2})}\ll \eta,
\end{align}
in which we used $\mathrm{dist}(k,\{a_k,b_k\})\geq m-1$.
A similar estimate holds for $|\tilde{G}^{E+i\eta}_{[a_k,b_k]}(\theta;b_k,k)|$ as well.
Combining \eqref{eq:G=tG_1_case1}, \eqref{eq:G=tG_2_case1} with \eqref{eq:G=tG_case1}, we conclude that
\begin{align}\label{eq:G=tG_n_sch}
    |G^{E+i\eta}_{[0,N-1]}(\theta;k,k)|\leq |\tilde{G}^{E+i\eta}_{[a_k,b_k]}(\theta;k,k)|+C.
\end{align}
Once again, by Cramer's rule and similar numerator/denominator estimates,
\begin{align}\label{eq:tG_est_sch}
    |\tilde{G}^{E+i\eta}_{[a_k,b_k]}(\theta;k,k)|
    =\frac{|D_{k-a_k}(\theta+a_k\omega,E+i\eta)|\cdot |D_{b_k-k}(\theta+(k+1)\omega,E+i\eta)|}{|D_{2m}(\theta+a_k\omega,E+i\eta)|} \leq e^{Cm^{1-\gamma/2}}.
\end{align}
Thus, the overall contribution from Case 1 can be bounded by:
\begin{align}
&\frac{2\eta}{N}\sum_k \chi_{(\tilde{\mathcal{B}}_{2m}^{\eta})^c}(\theta+k\omega-m\omega)\cdot \mathrm{Im}(G^{E+i\eta}_{[0,N-1]}(\theta;k,k))\\
&\leq \frac{2\eta}{N}\sum_k \chi_{(\tilde{\mathcal{B}}_{2m}^{\eta})^c}(\theta+k\omega-m\omega)\cdot\left(e^{Cm^{1-\gamma/2}}+C\right)\\
&\leq C\eta\cdot e^{Cm^{1-\gamma/2}}+C\eta\leq C\eta^{1-o(1)},
\end{align}
where we used $1-\gamma/2<\gamma_1$ and hence $\exp(C m^{1-\gamma/2})<\eta^{-o(1)}$ in the last step.
This completes the analysis for Case 1. 
\smallskip 

If $\theta+k\omega-m\omega\in \tilde{\mathcal{B}}_{2m}^{\eta}$, we further divide into Case 2 (intermediate resonant) and Case 3 (strong resonant). Case 2 will require the crucial local zero count and the delicate factorization of $D_{2m}$. 

\smallskip 

 {\bf Case 2.\ } $\theta+k\omega-m\omega\in \tilde{I}_j\setminus (\bigcup_{\ell=1}^{\ell_j}(I_{j,\ell}{{-k_j\omega}}))$ for some $j$.

 \medskip

We have by \eqref{eq:wj_lower_zeros},
\begin{align}\label{eq:Dtj_lower}
    &|D_{2m}(\theta+k\omega-m\omega+k_j\omega,E+i\eta)|\\
    &\qquad \geq \exp(2m(L(E,2m)-m^{-\gamma/2}))\cdot \prod_{\ell=1}^{\ell_j}\|\theta+k\omega-m\omega+k_j\omega-\theta_{j,\ell}\|.
\end{align}
Let $[a_k,b_k]:=[k-m+k_j, k-m+k_j+2m-1]$, and  
\begin{align}
    \tilde{H}_{\theta}|_{[0,N-1],k}:=H_{\theta}|_{[a_k,b_k]} \oplus H_{\theta}|_{[0,N-1]\setminus [a_k,b_k]}
\end{align}
Since $|k_j|<(1-\epsilon)m$,  
\begin{align}\label{eq:k_to_boundary}
    \mathrm{dist}(k, \{a_k,b_k\})\geq \epsilon m.
\end{align}
Let $\tilde{G}^{E+i\eta}_{[0,N-1],k}:=(\tilde{H}_{\theta}|_{[0,N-1],k}-E)^{-1}$.
By the resolvent identity, we have
\begin{align}\label{eq:G=tG}
    &G_{[0,N-1]}^{E+i\eta}(\theta;k,k)
    =\tilde{G}_{[0,N-1],k}^{E+i\eta}(\theta;k,k)\\
    &\qquad\qquad\qquad\qquad-G^{E+i\eta}_{[0,N-1]}(\theta;k,a_k-1)\cdot \tilde{G}^{E+i\eta}_{[a_k,b_k]}(\theta;a_k,k)\\
    &\qquad\qquad\qquad\qquad-G^{E+i\eta}_{[0,N-1]}(\theta;k,b_k+1)\cdot \tilde{G}^{E+i\eta}_{[a_k,b_k]}(\theta;b_k,k).
\end{align}
We bound the last two product terms above as follows:
\begin{align}\label{eq:G=tG_1}
    \max(|G^{E+i\eta}_{[0,N-1]}(\theta;k,a_k-1)|,\, |G^{E+i\eta}_{[0,N-1]}(\theta;k,b_k+1)|)\leq \eta^{-1},
\end{align}
and using \eqref{eq:Dtj_lower} and Cramer's rule, we have
\begin{align}\label{eq:G=tG_2}
    |\tilde{G}^{E+i\eta}_{[a_k,b_k]}(\theta;a_k,k)|&=\frac{|D_{b_k-k}(\theta+(k+1)\omega,E+i\eta)|}{|D_{2m}(\theta+a_k\omega,E+i\eta)|}\\
    &\qquad \lesssim \tau^{-2\kappa} e^{-\epsilon m(L(E)-o(1))}\ll \eta,
\end{align}
in which we used \eqref{eq:k_to_boundary} and $\|\theta+k\omega-m\omega+k_j\omega-\theta_{k,\ell}\|\geq \tau$.
A similar estimate holds for $|\tilde{G}^{E+i\eta}_{[a_k,b_k]}(\theta;b_k,k)|$ as well.
Combining \eqref{eq:G=tG_1}, \eqref{eq:G=tG_2} with \eqref{eq:G=tG}, we conclude that
\begin{align}\label{eq:G=tG_n_sch'}
    |G^{E+i\eta}_{[0,N-1]}(\theta;k,k)|\leq |\tilde{G}^{E+i\eta}_{[a_k,b_k]}(\theta;k,k)|+C.
\end{align}
Once again, by Cramer's rule, 
\begin{align}\label{eq:tG_est_sch'}
    &|\tilde{G}^{E+i\eta}_{[a_k,b_k]}(\theta;k,k)|\\
    &\qquad=\frac{|D_{k-a_k}(\theta+a_k\omega,E+i\eta)|\cdot |D_{b_k-k}(\theta+(k+1)\omega,E+i\eta)|}{|D_{2m}(\theta+a_k\omega,E+i\eta)|} \notag\\
    &\qquad\lesssim e^{Cm^{1-\gamma/2}}\cdot \prod_{\ell=1}^{\ell_j} \|\theta+k\omega-m\omega+k_j\omega-\theta_{j,\ell}\|^{-1}.
\end{align}
Thus, the overall contribution from Case 2 can be bounded by:
\begin{align}
&\frac{2\eta}{N}\sum_j\sum_k \chi_{\tilde{I}_j\setminus (\bigcup_{\ell}(I_{j,\ell}-k_j\omega))}(\theta+k\omega-m\omega)\cdot \mathrm{Im}(G^{E+i\eta}_{[0,N-1]}(\theta;k,k))\\
&\leq \frac{2\eta}{N}\sum_j\sum_k \chi_{\tilde{I}_j\setminus (\bigcup_{\ell}(I_{j,\ell}-k_j\omega))}(\theta+k\omega-m\omega)\cdot\\
&\qquad\qquad \cdot\left(e^{Cm^{1-\gamma/2}}\cdot \prod_{\ell=1}^{\ell_j}\|\theta+k\omega-m\omega+k_j\omega-\theta_{j,\ell}\|^{-1}+C\right)\\
&\lesssim e^{Cm^{1-\gamma/2}}\cdot \eta \int_{\tau}^1 \theta^{-2\kappa}\, \mathrm{d}\theta + C\eta\lesssim \eta^{1-o(1)} \tau^{1-2\kappa},
\end{align}
where we used $1-\gamma/2<\gamma_1$ and hence $\exp(Cm^{1-\gamma/2})<\eta^{-o(1)}$ in the last step. 

\smallskip

{\bf Case 3.\ }  $\theta+k\omega-m\omega\in I_{j,\ell}{{-k_j\omega}}$ for some $j,\ell$.

\medskip 

We use the trivial bound $|G_{[0,N-1]}^{E+i\eta}(\theta;k,k)|\leq \eta^{-1}$. Hence for $N\to \infty$, by the ergodic theorem,
\begin{align}
    &\frac{2\eta}{N} \sum_{j,\ell}\sum_k  \chi_{I_{j,\ell}-k_j\omega}(\theta+k\omega-m\omega) \cdot \mathrm{Im}(G_{[0,N-1]}^{E+i\eta}(\theta;k,k))\\
    &\qquad \leq \frac{2}{N}\sum_{j,\ell}\sum_k \chi_{I_{j,\ell}-k_j\omega}(\theta+k\omega-m\omega)\leq 2\sum_{j,\ell}|I_{j,\ell}|+\tau\lesssim m\tau\lesssim \eta^{\frac{1}{2k}-o(1)},
\end{align}
where we used $\tau=\eta^{\frac{1}{2k}}$, see \eqref{eq:choose_tau}.

\smallskip

Finally, combining the three cases and sending $N\to\infty$, we arrive at
\begin{align}
    d_{\omega,N}(\theta,E-\eta,E+\eta)\lesssim \eta^{1-o(1)}+\eta^{1-o(1)}\tau^{1-2\kappa}+\eta^{\frac{1}{2k}-o(1)}\lesssim \eta^{\frac{1}{2\kappa}-o(1)}, 
\end{align}
which  is the claimed H\"older regularity. \qed

\section{Analysis of local zeros for the Jacobi block case}
\label{sec:block1}

Recall that in the Jacobi block case, $f_n(z,E)$ is the finite-volume determinant with periodic boundary condition, see \eqref{def:fn}. 
$E_0\in \R$ is a fixed energy for which $L_d(\omega,E_0)>0$.
Throughout Sections \ref{sec:block1} and \ref{sec:block2}, we write for simplicity $\kappa^d=\kappa^d(\omega,E_0)$. The energy $E$ will be assumed to lie in a neighborhood $I_{E_0}$ of $E_0$, see \eqref{def:IE0}.

\begin{lemma}\label{lem:local_zero*}
Let $n$ be $\kappa_0$-admissible and large. 
For each ball $B(z_0,r_n)$, with $z_0\in \kreis^1$ and $r\simeq r_n:=e^{-(\log n)^{C_0}}$ with $C_0>1$, there exists $|k|<(1-\epsilon)n/2$ such that $f_n(e^{2\pi ik\omega} z,E)$ has at most $2\kappa^d$ zeros in $B(z_0,r)$.
\end{lemma}
\begin{proof}
    Assume that $f_n(e^{2\pi ik\omega}z,E)$ has at least $2\kappa^d+1$ zeros in $B(z_0,r)$ for each $|k|<(1-\epsilon)n/2$.
    This implies that for each such $k$ that $f_n(z,E)$ has at least $2\kappa+1$ zeros in $B(e^{-2\pi ik\omega}z_0,r)$.
    These balls are disjoint from each other due to the Diophantine condition on $\omega$: 
    \[\|(k_1-k_2)\omega\|\gtrsim \frac{1}{n^A}\gg e^{-(\log n)^{C_0}}.\]
    Therefore, $f_n(z,E)$ has at least $(2\kappa^d+1)(1-\epsilon)n>2\kappa^d (1+o(1)) n$ zeros in the annulus, which contradicts Theorem~\ref{thm:acc=zeros*} and $\kappa^d(\omega,E)\leq \kappa^d$ according to \eqref{def:IE0}.
\end{proof}

\begin{lemma}\label{lem:un_lower_zeros*}
Let $n$ be $\kappa_0$-admissible. 
For each $z_0\in \kreis^1$, let $k_0=k_0(z_0)$, $|k_0|<(1-\epsilon)n/2$ be as in Lemma~\ref{lem:local_zero*}.  
Then there exists an integer $C_0\in [1,2\kappa+1]$ such that $f_n(e^{2\pi i k_0\omega}z,E)$ has no zero in $B(z_0,4(C_0+1)r_n)\setminus B(z_0,4C_0r_n)$, and has at most $2\kappa^d$ zeros, denoted by $\{z_{n,1},...,z_{n,j_0}\}$, in $B(z_0,4C_0r_n)$. Furthermore, for $z\in B(z_0,(4C_0+3)r_n)$,  the following lower bound holds:
\begin{align}\label{eq:un_lower_zeros*}
    \log |f_n(e^{2\pi ik_0\omega}z,E)|\geq n(L^d(E,n)+\langle \log |\det B(\cdot)|\rangle-n^{-\gamma/2})+\sum_{j=1}^{j_0}\log|z-z_{n,j}|.
\end{align}
\end{lemma}
\begin{proof}
    The existence of $C_0$ follows from Lemma \ref{lem:local_zero*}. 
    The function 
    \begin{align}
        Y_n(z):=\frac{f_n(e^{2\pi ik_0\omega}z,E)}{\prod_{j=1}^{j_0} (z-z_{n,j})}
    \end{align}
     has no zeros in $B(z_0, 4(C_0+1)r_n)$, whence $y_n(z):=n^{-1}\log |Y_n(z)|$ is harmonic  in $B(z_0,4(C_0+1)r_n)$.
    For the upper bound on $y_n(z)$ on $\partial B(z_0,4(C_0+1)r_n)$ we note that  
    if $z\in \partial B(z_0,4(C_0+1)r_n)$, then for each $1\leq j\leq j_0$, one has $|z-z_{n,j}|\geq 4r_n$. This implies that
    \begin{align}
        \prod_{j=1}^{j_0} |z-z_{n,j}|\geq (4r_n)^{j_0}\geq (4r_n)^{2\kappa}\geq e^{-3\kappa(\log n)^{C_0}}.
    \end{align}
    By Lemma \ref{lem:upperbd_E'} the following upper bound holds
    \begin{align*}
    n^{-1}\log|f_n(e^{2\pi ik_0\omega}z,E)| & \leq \sup_{|\varepsilon'|\leq 4C_0r_n} \big(L_{\varepsilon'}^d(E)+\langle \log |\det B(\cdot+i\varepsilon')|\rangle\big)+2n^{-\gamma}\\
    & \leq L^d(E)+\langle \log |\det B(\cdot)|\rangle+3n^{-\gamma}
    \end{align*}
    Therefore, for any $z\in \partial B(z_0,4(C_0+1)r_n)$ and hence for all $z\in B(z_0,4(C_0+1)r_n)$ by the maximum principle, 
    \begin{align}
        y_n(z)\leq L^d(E)+\langle \log |\det B(\cdot)|\rangle+3n^{-\gamma}+3\kappa \frac{(\log n)^{C_0}}{n}<L^d(E)+\langle \log |\det B(\cdot)|\rangle+4n^{-\gamma}.
    \end{align}
    Applying Harnack's inequality to the non-negative harmonic function \[\tilde{y}_n(z):=L^d(E)+\langle \log |\det B(\cdot)|\rangle+4n^{-\gamma}-y_n(z)\] on $B(z_0, (4C_0+3)r_n)$ we infer that
    \begin{align}
        \sup_{z\in B(z_0,(4C_0+3)r_n)} \tilde{y}_n(z)\leq C \inf_{z\in B(z_0,(4C_0+3)r_n)}\tilde{y}_n(z).
    \end{align}
    The large deviation estimate applied to $f_n(e^{2\pi i \theta},E)$ yields
    \begin{align*}
        & \mathrm{mes}\{\theta: n^{-1}\log |f_n(e^{2\pi i\theta},E)|<L^d(E,n)+\langle \log |\det B(\cdot)|\rangle-n^{-\gamma}\}\leq e^{-n^{\gamma}}\\
        & \qquad\qquad\qquad\qquad\qquad\qquad\qquad\qquad\qquad\qquad\ll r_n\simeq \diam (B(z_0,r_n)\cap \kreis^1).
    \end{align*}
    Hence there exists $z_1\in B(z_0,r_n)$ such that
    \begin{align}
        y_n(z_1)>L^d(E,n)+\langle \log |\det B(\cdot)|\rangle-n^{-\gamma},
    \end{align}
    whence
    \begin{align}
        \inf_{z\in B(z_0,r_n)}\tilde{y}_n(z)\leq \tilde{y}_n(z_1)\leq 5n^{-\gamma}.
    \end{align}
    This implies that for all $z\in B(z_0,(4C_0+3)r_n)$ 
    \begin{align}
        y_n(z)\geq L^d(E,n)+\langle \log |\det B(\cdot)|\rangle-Cn^{-\gamma}>L^d(E,n)+\langle \log |\det B(\cdot)|-n^{-\gamma/2}.
    \end{align}
    This completes the proof. \end{proof}

    \begin{lemma}\label{lem:un_lower_zeros_eta*}
    Let $n,z_0,k_0,C_0$ be as in Lemma \ref{lem:un_lower_zeros*}. Let $\eta\in \R$ such that $|\eta|\leq Ce^{-n^{\gamma_1}}$ with $\gamma_1$ as in \eqref{def:gamma1}. 
    Then $f_n(e^{2\pi ik_0\omega}z,E+i\eta)$ has the same number of zeros as $f_n(e^{2\pi ik_0\omega}z,E)$ in $B(z_0,(4C_0+2\pm 1)r_n)$, with the zeros denoted by $\{z_{n,1}^{\eta},...,z_{n,j_0}^{\eta}\}$. Furthermore for any $z\in B(z_0,(4C_0+2)r_n)$, we have the following lower bound:
    \begin{align}\label{eq:un_lower_zeros_eta*}
    \log |f_n(e^{2\pi ik_0\omega}z,E+i\eta)|\geq n(L^d(E,n)+\langle \log |\det B(\cdot)|\rangle-n^{-\gamma/2})+\sum_{j=1}^{j_0}\log|z-z_{n,j}^{\eta}|.
\end{align}
\end{lemma}
\begin{proof}
The idea behind the zero count is to apply  Rouch\'e's theorem.
For $z\in B(z_0,(4C_0+2\pm 1)r_n)$, consider the following difference and apply the standard telescoping argument, we have, with $z=e^{2\pi i(\theta+i\varepsilon)}$,
\begin{align}\label{eq:tele_wedge_j}
    &\|\textstyle{\bigwedge^j}M_{n,E+i\eta}(\theta+k_0\omega+i\varepsilon)-\textstyle{\bigwedge^j}M_{n,E}(\theta+k_0\omega+i\varepsilon)\|\\
    \leq &\sum_{\ell=1}^n \|\textstyle{\bigwedge^j}M_{n-\ell,E+i\eta}(\theta+(k_0+\ell+1)\omega+i\varepsilon)\|\cdot \|\textstyle{\bigwedge^j}M_{E+i\eta}(\theta+(k_0+\ell)\omega+i\varepsilon)-\textstyle{\bigwedge^j}M_E(\theta+(k_0+\ell)\omega+i\varepsilon)\|\\
    &\qquad \cdot \|\textstyle{\bigwedge^j} M_{\ell-1,E}(\theta+k_0\omega+i\varepsilon)\|.
\end{align}

For $j=d$, by Lemma \ref{lem:upperbd_E'}, and  $|\varepsilon|\lesssim r_n$, we have, uniformly in $\theta\in\T$,
\begin{align}\label{eq:M_n-ell_d}
    \|\textstyle{\bigwedge^d}M_{n-\ell,E+i\eta}(\theta+(k_0+\ell+1)\omega+i\varepsilon)\|
    \leq &\exp((n-\ell)L^d_{\varepsilon}(E)+2n^{1-\gamma})\\
    \leq &\exp((n-\ell)L^d(E)+3n^{1-\gamma}),
\end{align}
and
\begin{align}\label{eq:M_ell-1_d}
    \|\textstyle{\bigwedge^d} M_{\ell-1,E}(\theta+k_0\omega+i\varepsilon)\|
    \leq \exp((\ell-1)L^d(E)+3n^{1-\gamma}).
\end{align}

For $j\in \{1,...,2d\}\setminus \{d\}$, by Lemma \ref{lem:upper_j_neq_d}, and $|\varepsilon|\lesssim r_n$, we have, uniformly in $\theta\in \T$,
\begin{align}\label{eq:M_n-ell_j}
    \|\textstyle{\bigwedge^j}M_{n-\ell,E+i\eta}(\theta+(k_0+\ell+1)\omega+i\varepsilon)\|
    \leq \exp((n-\ell)(L^j(E)+o(1))),
\end{align}
and 
\begin{align}\label{eq:M_ell-1_j}
    \|\textstyle{\bigwedge^j} M_{\ell-1,E}(\theta+k_0\omega+i\varepsilon)\|
    \leq \exp((\ell-1)(L^j(E)+o(1))).
\end{align}

Concerning the difference term  we have
\begin{align}\label{eq:Mj-Mj}
    &\|\textstyle{\bigwedge^j}M_{E+i\eta}(\theta+(k_0+\ell)\omega+i\varepsilon)-\textstyle{\bigwedge^j}M_E(\theta+(k_0+\ell)\omega+i\varepsilon)\|\\
    &\qquad \leq C_d \|M_{E+i\eta}(\theta+(k_0+\ell)\omega+i\varepsilon)-M_{E}(\theta+(k_0+\ell)\omega+i\varepsilon)\|\leq C_d \eta.
\end{align}

Combining \eqref{eq:M_n-ell_d}, \eqref{eq:M_ell-1_d}, \eqref{eq:Mj-Mj} above with \eqref{eq:tele_wedge_j}, we obtain
\begin{align}\label{eq:wj_Mn-Mn_d}
&\sup_{\theta\in \T} \|\textstyle{\bigwedge^d}M_{n,E+i\eta}(\theta+k_0\omega+i\varepsilon)-\textstyle{\bigwedge^d}M_{n,E}(\theta+k_0\omega+i\varepsilon)\|\notag\\
    &\qquad \leq C_d n \eta \cdot \exp(nL^d(E)+6n^{1-\gamma})\notag\\
    &\qquad \leq \exp(n L^d(E)-n^{\gamma_1}/2),
\end{align}
where we used $\gamma_1>1-\gamma$ in the last inequality.

Similarly, for $j\neq d$, combining \eqref{eq:M_n-ell_j}, \eqref{eq:M_ell-1_j}, \eqref{eq:Mj-Mj} above with \eqref{eq:tele_wedge_j}, we obtain
\begin{align}\label{eq:wj_Mn-Mn_j}
&\sup_{\theta\in \T} \|\textstyle{\bigwedge^j}M_{n,E+i\eta}(\theta+k_0\omega+i\varepsilon)-\textstyle{\bigwedge^j}M_{n,E}(\theta+k_0\omega+i\varepsilon)\|\notag\\
    &\qquad \leq C_d n \eta \cdot \exp(n(L^j(E)+o(1)))\notag\\
    &\qquad \leq \exp(n (L^j(E)+o(1))).
\end{align}

Taking into account that $L^d(E)>L^j(E)+o(1)$ for any $j\neq d$, we have
by \eqref{eq:wj_Mn-Mn_d} and \eqref{eq:wj_Mn-Mn_j} that
\begin{align}\label{eq:wj_Mn-Mn_sup}
    \sup_{j=0}^{2d} \sup_{\theta\in \T} \|\textstyle{\bigwedge^j}M_{n,E+i\eta}(\theta+k_0\omega+i\varepsilon)-\textstyle{\bigwedge^j}M_{n,E}(\theta+k_0\omega+i\varepsilon)\|\leq \exp(nL^d(E)-n^{\gamma_1}/2).
\end{align}

Before we move on, let us denote $M_{n,E}(\theta+k_0\omega+i\varepsilon)=:A$, $M_{n,E+i\eta}(\theta+k_0\omega+i\varepsilon)=:\tilde{A}$, and $\prod_{j=0}^{n-1} \det(B(\theta+(k_0+j)\omega+i\varepsilon))=:b_n(\theta+i\varepsilon)$ for simplicity. 
Consider the difference:
\begin{align}\label{eq:fn-fn}
    &|f_n(e^{2\pi ik_0\omega}z,E)-f_n(e^{2\pi ik_0\omega}z,E+i\eta)|\notag\\
    &=|b_n(\theta+i\varepsilon)|\cdot |\det(A-I_{2d})-\det(\tilde{A}-I_{2d})|\notag\\
    &=|b_n(\theta+i\varepsilon)|\cdot\left|\langle e_1\textstyle{\wedge}\cdots \textstyle{\wedge} e_{2d}, \, \textstyle{\bigwedge^{2d}}(A-I_{2d}) \, (e_1\textstyle{\wedge} \cdots \textstyle{\wedge} e_{2d})-\textstyle{\bigwedge^{2d}}(\tilde{A}-I_{2d}) \, (e_1\textstyle{\wedge} \cdots \textstyle{\wedge} e_{2d})\rangle\right|\notag\\
    &\leq |b_n(\theta+i\varepsilon)|\cdot \|\textstyle{\bigwedge^{2d}}(A-I_{2d}) \, (e_1\textstyle{\wedge} \cdots \textstyle{\wedge} e_{2d})-\textstyle{\bigwedge^{2d}}(\tilde{A}-I_{2d}) \, (e_1\textstyle{\wedge} \cdots \textstyle{\wedge} e_{2d})\|.
\end{align}
Note that for arbitrary $2d\times 2d$ matrices $D_1,D_2$, we have
\begin{align}
    \textstyle{\bigwedge^{2d}} (D_1-D_2)\, (e_1\textstyle{\wedge}\cdots \textstyle{\wedge}e_{2d})=\sum_{S\subset \{1,...,2d\}} \textstyle{\bigwedge}_{i=1}^{2d} w_i^{(S)}, 
\end{align}
where 
\begin{align}
    w_i^{(S)}=\begin{cases}
        D_1e_i, \text{ if } i\in S,\\
        -D_2 e_i, \text{ if } i\notin S.
    \end{cases}
\end{align}
This implies
\begin{align}
    \textstyle{\bigwedge^{2d}}(D_1-D_2)\, (e_1\textstyle{\wedge}\cdots \textstyle{\wedge} e_{2d})=\sum_{S\subset \{1,...,2d\}} (-1)^{m(S)} (\textstyle{\bigwedge}_{i\in S} D_1e_i)\, \textstyle{\bigwedge} \, (\textstyle{\bigwedge}_{i\notin S} (-D_2e_i)),
\end{align}
where $m(S)\in \{0,1\}$ is uniquely determined by $S$. Applying this formula to $D_1=A$ (and similarly to $D_1=\tilde{A}$) and $D_2=I_{2d}$, we arrive at
\begin{align}
    \textstyle{\bigwedge^{2d}}(A-I_{2d}) \, (e_1\textstyle{\wedge} \cdots \textstyle{\wedge} e_{2d})=\sum_{S\subset \{1,...,2d\}} (-1)^{\tilde{m}(S)} (\textstyle{\bigwedge}_{i\in S} Ae_i)\, \textstyle{\bigwedge}\, (\textstyle{\bigwedge}_{i\notin S} e_i),
\end{align}
with $\tilde{m}(S)\in \{0,1\}$ uniquely determined by $S$, and a similar expression for $\tilde{A}$. Taking the difference of the expression above between $A$ and $\tilde{A}$, we obtain
\begin{align}
    &\textstyle{\bigwedge^{2d}}(A-I_{2d}) \, (e_1\textstyle{\wedge} \cdots \textstyle{\wedge} e_{2d})-\textstyle{\bigwedge^{2d}}(\tilde{A}-I_{2d}) \, (e_1\textstyle{\wedge} \cdots \textstyle{\wedge} e_{2d})\\
    &=\sum_{S\subset \{1,...,2d\}} (-1)^{\tilde{m}(S)} \left(\textstyle{\bigwedge^{|S|}}A-\textstyle{\bigwedge^{|S|}}\tilde{A}\right)(\textstyle{\bigwedge}_{i\in S}e_i)\, \textstyle{\bigwedge}\,  (\textstyle{\bigwedge}_{i\notin S}e_i).
\end{align}
Plugging this back into \eqref{eq:fn-fn}, we have
\begin{align}\label{eq:fn-fn_2}
    &|f_n(e^{2\pi ik_0\omega}z,E)-f_n(e^{2\pi ik_0\omega}z,E+i\eta)|\notag\\
   & \leq C_d\, |b_n(\theta+i\varepsilon)|\cdot \sup_{j=0}^{2d} \|\textstyle{\bigwedge^j}M_{n,E+i\eta}(\theta+k_0\omega+i\varepsilon)-\textstyle{\bigwedge^j}M_{n,E}(\theta+k_0\omega+i\varepsilon)\|\notag\\
   & \leq C_d\, |b_n(\theta+i\varepsilon)|\cdot \exp(n L^d(E)-n^{\gamma_1}/2)),
\end{align}
where we used \eqref{eq:wj_Mn-Mn_sup} in the last line. We further bound
\begin{align}
    \sup_{\theta\in \T}|b_n(\theta+i\varepsilon)|
   & \leq \exp(n \langle \log | \det B(\cdot+i\varepsilon)|\rangle +n^{o(1)})\\
   & \leq \exp(n \langle \log | \det B(\cdot)|\rangle +2n^{o(1)}),
\end{align}
where we used $|\varepsilon|\lesssim r_n$ in the last estimate.
Plugging this back into \eqref{eq:fn-fn_2}, we obtain
\begin{align}\label{eq:fn-fn_3}
    |f_n(e^{2\pi ik_0\omega}z,E)-f_n(e^{2\pi ik_0\omega}z,E+i\eta)|\leq \exp(n (L^d(E)+\langle \log |\det B(\cdot)|\rangle-n^{\gamma_1-1}/4)).
\end{align}
Lemma~\ref{lem:un_lower_zeros*} implies that for $z\in \partial B(z_0,(4C_0+2\pm 1)r_n)$,
\begin{align}\label{eq:fn_lower}
    |f_n(e^{2\pi ik_0\omega}z,E)|
    &\geq \exp(n(L^d(E,n)+\langle \log |\det B(\cdot)|\rangle-n^{-\gamma/2}))\cdot r_n^{2\kappa^d}\\
   & \geq \exp(n(L^d(E)+\langle \log |\det B(\cdot)|\rangle-2n^{-\gamma/2})).
\end{align}
Therefore, since $\gamma_1>1-\gamma/2$, we have by \eqref{eq:fn-fn_3} and \eqref{eq:fn_lower} that
\begin{align}
    |f_n(e^{2\pi ik_0\omega}z,E)-f_n(e^{2\pi ik_0\omega}z,E+i\eta)|<\frac{1}{2}|f_n(e^{2\pi ik_0\omega}z,E)|.
\end{align}
The proof of \eqref{eq:un_lower_zeros_eta*} then follows along the same lines as those leading to \eqref{eq:un_lower_zeros} in Lemma \ref{lem:un_lower_zeros}.
\end{proof}

\section{H\"older continuity in the block-valued case}
\label{sec:block2}

The proof is analogous to the scalar-valued case. 
Our goal is to estimate the following expression for arbitrary fixed $E\in I_{E_0}$ and $\theta\in \T$ as $N\to\infty$:
\begin{align}\label{eq:goal*}
    &d_{\omega,dN}(\theta,E-\eta,E+\eta)=\frac{1}{dN}\mathrm{tr}(P_{[E-\eta, E+\eta)}(H_{\omega,\theta}|_{[0,dN-1]})),
\end{align}
cf.\ \eqref{eq:goal}. As before, we 
bound~\eqref{eq:goal*} through the trace of the Green's function: 
\begin{align}\label{eq:dN<G*}
    d_{\omega,dN}(\theta,E-\eta,E+\eta)
    &\leq \frac{2\eta}{dN} \mathrm{Im}(\mathrm{tr}(G_{[0,dN-1]}^{E+i\eta}(\theta)))\\
    &=\frac{2\eta}{dN}\sum_{k=0}^{dN-1} \mathrm{Im} (G_{[0,dN-1]}^{E+i\eta}(\theta;k,k)),
\end{align}
in which $G^{E+i\eta}_{[0,dN-1]}(\theta)=(H_{\omega,\theta}|_{[0,dN-1]}-(E+i\eta))^{-1}$ is the kernel of the resolvent.
Let the auxiliary parameter $\tau$ be defined as 
\begin{align}\label{def:tau_d}
    \tau=\eta^{\frac{1}{2\kappa^d}},
\end{align}
in this block-valued case.
For any $\kappa_0$-admissible $n\in \N$, let 
\begin{align}
    \mathcal{B}_{n,h}^{\eta}:=\{\theta\in \T: u_n(e^{2\pi i\theta},E+i\eta)<L^d(E,n)+\langle \log |\det B(\cdot)|\rangle-h\cdot n^{-\gamma}\}
\end{align}
be the large deviation set for $u_n(z,E+i\eta)=n^{-1}\log |f_n(z,E+i\eta)|$. 
We choose $m\in \N$ such that $2m$ is $\kappa_0$-admissible and $e^{-(2m)^{\gamma_1}}\simeq \eta$. 
The set $\mathcal{B}_{2m,2}^{\eta}$ satisfies a large deviation estimate
\begin{align}\label{eq:B2m_eta_block}
    \mes\,\mathcal{B}_{2m,2}^{\eta}\leq e^{-(2m)^{\gamma}}.
\end{align}
In fact, note $2m$ is $\kappa_0$ admissible, the large deviation estimate of Lemma \ref{lem:deno} applies to $\mathcal{B}_{2m,1}^0$ implying
\begin{align}\label{eq:B2m_0_block}
    \mathrm{mes}\,\mathcal{B}_{2m,1}^0\leq e^{-(2m)^{\gamma}}.
\end{align}
Suppose $\theta\in (\mathcal{B}_{2m,1}^0)^c$, then
\begin{align}
    |f_{2m}(e^{2\pi i\theta},E)|\geq \exp(2m(L^d(E,2m)+\langle \log |\det B(\cdot) |\rangle-(2m)^{-\gamma})).
\end{align}
By the same telescoping argument in the proof of Lemma \ref{lem:un_lower_zeros_eta*}, using, $\gamma_1>1-\gamma/2$, we have
\begin{align}
    |f_{2m}(e^{2\pi i\theta},E+i\eta)|\geq \frac{1}{2} |f_{2m}(e^{2\pi i\theta}, E)|\geq \exp(2m(L^d(E,2m)+\langle \log |\det B(\cdot) |\rangle-2(2m)^{-\gamma})).
\end{align}
This implies $(\mathcal{B}_{2m,1}^0)^c\subset (\mathcal{B}_{2m,2}^{\eta})^c$, hence the estimate \eqref{eq:B2m_eta_block} follows from \eqref{eq:B2m_0_block}.
Let $\{\xi_j\, e^{2\pi i\theta_j}\}_{j=1}^{J_0}$, $\xi_j>0$ and $\theta_j\in \T$, be the zeros of $f_{2m}(z,E+i\eta)$. 
Cartan's estimate implies that 
\begin{align}\label{eq:B2m_theta_pm_rm*}
\mathcal{B}_{2m,m^{\gamma/2}}^{\eta}\subset \bigcup_{j=1}^{J_0}(\theta_j-\tilde{r}_m, \theta_j+\tilde{r}_m),
\end{align}
with $\tilde{r}_m:=e^{-cm^{\gamma/2}}$ for some constant $c>0$, and $J_0\lesssim m$.
Applying Lemma~\ref{lem:un_lower_zeros_eta*} to $n=2m$ and $z_0=e^{2\pi i\theta_j}$, for each $1\leq j\leq J_0$, yields the existence of $|k_j|<(1-\varepsilon)m$, $C_j\in [1,2\kappa^d+1]$, and $\ell_j\in [1,2\kappa^d]$, such that $w_j(z):=f_{2m}(z e^{2\pi ik_j\omega},E+i\eta)$
has $\ell_j$ zeros in $B(e^{2\pi i\theta_j}, (4C_j+1)r_m)$, $r_m\simeq \exp(-(\log m)^{C_0})\gg \tilde{r}_m$.
Denoting the zeros of $f_{2m}(z e^{2\pi ik_j\omega},E+i\eta)$ by $\{\xi_{j,\ell}\, e^{2\pi i\theta_{j,\ell}}\}_{\ell=1}^{\ell_j}$ by $\xi_{j,\ell}>0$, $\theta_{j,\ell}\in \T$, one has
\begin{align}\label{eq:wj_lower_zeros*}
    \log |w_j(z)|\geq 2m(L^d(E,2m)+\langle \log |\det B(\cdot)|\rangle-m^{-\gamma/2})+\sum_{\ell=1}^{\ell_j}\log |z-\xi_{j,\ell}\, e^{2\pi i\theta_{j,\ell}}|,
\end{align}
for all $z\in B(e^{2\pi i\theta_j},(4C_j+2)r_m)$.
For each $1\leq j\leq J_0$ and $1\leq \ell\leq \ell_j$, let 
   $  \tilde{I}_j, 
    I_{j,\ell}$ be defined as in~\eqref{eq:Ijdef}. 
With $\tau$ as in \eqref{eq:choose_tau} and $\tau\ll r_m$, we have for each $j,\ell$ that $I_{j,\ell}\subset \tilde{I}_j$.
As before, we let $\tilde{\mathcal{B}}_{2m}^{\eta}:=\bigcup_{j=1}^{J_0} \tilde{I}_j$. 
Clearly, by \eqref{eq:B2m_theta_pm_rm*}, $\mathcal{B}_{2m,m^{\gamma/2}}^{\eta}\subset \tilde{\mathcal{B}}_{2m}^{\eta}$.
We again divide the analysis of $G^{E+i\eta}_{[0,dN-1]}(\theta; k,k)$ into three cases.

\medskip

{\bf Case 1.} If $\theta+[k/d]\omega-m\omega\in (\tilde{\mathcal{B}}_{2m}^{\eta})^c$.

\smallskip

Since $\theta+[k/d]\omega-m\omega\in (\tilde{\mathcal{B}}_{2m}^{\eta})^c\subset (\mathcal{B}^{\eta}_{2m,m^{\gamma/2}})^c$, we have
\begin{align}\label{eq:f2m_lower_block_1}
   |f_{2m}(\theta+[k/d]\omega-m\omega,E+i\eta)|
   \geq &\exp(2m(L^d(E,2m)+\langle \log |\det B(\cdot)|\rangle -m^{\gamma/2}\cdot (2m)^{-\gamma}))\notag\\
   \geq &\exp(2m(L^d(E,2m)+\langle \log |\det B(\cdot)|\rangle -m^{-\gamma/2})).
\end{align}
Simply let $[a_k,b_k]:=[[k/d]d-md,[k/d]d+md-1]$ and define
\begin{align}
    \tilde{H}_{\theta}|_{[0,dN-1],k}:=H_{\theta}^p|_{[a_k,b_k]} \oplus H_{\theta}|_{[0,dN-1]\setminus [a_k,b_k]},
\end{align}
where $H_{\theta}^p|_{[a_k,b_k]}$ is the operator restricted to $[a_k,b_k]$ with periodic boundary condition, see \eqref{def:H_per}.
Let $  \Gamma_{[a_k,b_k]}
    =H_{\theta}|_{[0,dN-1]}-\tilde{H}_{\theta}|_{[0,dN-1],k}$
\begin{align}
    =&\left(\begin{array}{ c c|c c c c c|c c}
 & & & & & & & & \\
 & &B^{(*)}(\theta+(b_k+1)\omega/d) & & & & & &\\
\hline
 &B(\theta+(b_k+1)\omega/d) & & & & &-B(\theta+a_k\omega/d) & &\\
 & & & & & & & &\\
 & & & & & & & &\\
 & & & & & & & &\\
 & &-B^{(*)}(\theta+a_k\omega/d) & & & & &B^{(*)}(\theta+a_k\omega/d) &\\
\hline
 & & & & & &B(\theta+a_k\omega/d) & &\\
 & & & & & & & &
\end{array}\right),
\end{align}
where the square in the middle is the box $[a_k,b_k]$.
Let $\tilde{G}^{E+i\eta}_{[0,dN-1],k}(\theta):=(\tilde{H}_{\theta}|_{[0,dN-1],k}-E)^{-1}$.
By the resolvent identity, we have
\begin{align}\label{eq:G=tG*_1}
    &G_{[0,dN-1]}^{E+i\eta}(\theta;k,k)\\
    &\qquad=\tilde{G}_{[0,dN-1],k}^{E+i\eta}(\theta;k,k)
   -\left(G^{E+i\eta}_{[0,dN-1]}(\theta)\Gamma_{[a_k,b_k]}\tilde{G}^{E+i\eta}_{[0,dN-1],k}(\theta)\right)(k,k).
\end{align}
Within the product term $\left(G^{E+i\eta}_{[0,dN-1]}(\theta)\Gamma_{[a_k,b_k]}\tilde{G}^{E+i\eta}_{[0,dN-1],k}(\theta)\right)(k,k)$, we bound
\begin{align}\label{eq:G=tG_1*_1}
    \sup_{x\in \Z} |G^{E+i\eta}_{[0,dN-1]}(\theta;k,x)|\leq \eta^{-1}.
\end{align}
Hence
\begin{align}\label{eq:GGG_1}
    &\left|\left(G^{E+i\eta}_{[0,dN-1]}(\theta)\Gamma_{[a_k,b_k]}\tilde{G}^{E+i\eta}_{[0,dN-1],k}(\theta)\right)(k,k)\right|\\
    &\qquad \qquad \leq \sum_{\substack{x\in \Z\\ y\in [a_k,b_k]}} \eta^{-1} |\Gamma_{[a_k,b_k]}(x,y)|\cdot |\tilde{G}^{E+i\eta}_{[0,dN-1],k}(\theta;y,k)|.
\end{align}
In order to ensure $\Gamma_{[a_k,b_k]}(x,y)\neq 0$ while $y\in [a_k,b_k]$, it is necessary that 
\begin{align}
    &y\in [a_k,a_k+d-1]\cup [b_k-d+1,b_k], \text{ and }\\
    &x\in [a_k-d,a_k+d-1] \cup [b_k-d+1,b_k+d].
\end{align}
Thus we can further bound \eqref{eq:GGG_1} by
\begin{align}\label{eq:GGG_2}
    &\left|\left(G^{E+i\eta}_{[0,dN-1]}(\theta)\Gamma_{[a_k,b_k]}\tilde{G}^{E+i\eta}_{[0,dN-1],k}(\theta)\right)(k,k)\right|\\
    &\qquad \qquad \leq C_Bd^2\, \eta^{-1} \sup_{y\in [a_k,a_k+d-1]\cup [b_k-d+1,b_k]} |\tilde{G}^{E+i\eta}_{[0,dN-1],k}(\theta;y,k)|,
\end{align}
here $C_B$ is a constant depending only on $B(\cdot)$.
Also, by \eqref{eq:mufn}, we have
\begin{align}\label{eq:G=tG_2*_1}
    |\tilde{G}^{E+i\eta}_{[0,dN-1],k}(\theta;y,k)|=|\tilde{G}^{E+i\eta}_{[a_k,b_k]}(\theta;y,k)|
    &=\frac{|\mu^{E+i\eta}_{[a_k,b_k],y,k}(\theta)|}{|f_{2m}(\theta+a_k\omega/d,E+i\eta)|}.
\end{align}
By the upper bound of the numerator in Lemma \ref{lem:numerator} and lower bound of denominator in \eqref{eq:f2m_lower_block_1}, we have
\begin{align}\label{eq:tG_11}
    &|\tilde{G}^{E+i\eta}_{[a_k,b_k]}(\theta;y,k)|\\
    &\qquad\qquad\leq C\exp(m(L^{d-1}(E+i\eta,2m)+L^d(E+i\eta,2m)-2L^d(E,2m)+o(1))).
\end{align}
Note that since $|\eta|\leq e^{-cm^{\gamma_1}}$, Lemma \ref{lem:upperbd_E'} implies that 
\begin{align}
    L^d(E+i\eta,2m)\leq L^d(E,2m)+Cm^{-\gamma},
\end{align}
and Lemma \ref{lem:upper_j_neq_d} implies 
\begin{align}
    L^{d-1}(E+i\eta,2m)\leq L^{d-1}(E)+o(1).
\end{align}
Plugging these two estimates into \eqref{eq:tG_11} implies
\begin{align}\label{eq:G=tG_2*_2}
    |\tilde{G}^{E+i\eta}_{[a_k,b_k]}(\theta;y,k)| \leq C\exp(-m(L_d(E)-o(1)))\ll \eta.
\end{align}
Combining \eqref{eq:GGG_2}, \eqref{eq:G=tG_2*_2} with \eqref{eq:G=tG*_1}, we have
\begin{align}\label{eq:G=tG_n*_block}
    |G^{E+i\eta}_{[0,dN-1]}(\theta;k,k)|\leq |\tilde{G}^{E+i\eta}_{[a_k,b_k]}(\theta;k,k)|+C.
\end{align}
Once again, by Cramer's rule, and by the upper bound for numerator of Green's function in Lemma~\ref{lem:numerator_diag} and lower bound of denominator in \eqref{eq:f2m_lower_block_1}, we have
\begin{align*}\label{eq:tG_est*_block}
    |\tilde{G}^{E+i\eta}_{[a_k,b_k]}(\theta;k,k)|
    &=\frac{\mu_{[a_k,b_k],k,k}^{E+i\eta}(\theta)}{|f_{2m}(\theta+a_k\omega/d,E+i\eta)|} 
    \lesssim e^{Cm^{1-\gamma/2}}.
\end{align*}
Thus, the overall contribution from Case 1 can be bounded by:
\begin{align}
&\frac{2\eta}{dN}\sum_k \chi_{(\tilde{\mathcal{B}}_{2m}^{\eta})^c}(\theta+[k/d]\omega-m\omega)\cdot \mathrm{Im}(G^{E+i\eta}_{[0,N-1]}(\theta;k,k))\\
&\leq \frac{2\eta}{dN}\sum_k \chi_{(\tilde{\mathcal{B}}_{2m}^{\eta})^c}(\theta+[k/d]\omega-m\omega)\cdot\left(e^{Cm^{1-\gamma/2}}+C\right)\\
&\leq C\eta\cdot e^{Cm^{1-\gamma/2}}+C\eta\leq C\eta^{1-o(1)},
\end{align}
where we used $1-\gamma<\gamma_1$ and hence $\exp(C m^{1-\gamma})<\eta^{-o(1)}$ in the last step.
This completes the analysis for Case 1.

\medskip

{\bf Case 2.} If $\theta+[k/d]\omega-m\omega\in \tilde{I}_j\setminus (\bigcup_{\ell=1}^{\ell_j}(I_{j,\ell}{{-k_j\omega}}))$ for some $j$.

\smallskip

We have by \eqref{eq:wj_lower_zeros*},
\begin{align}\label{eq:Dtj_lower*}
    &\qquad |f_{2m}(\theta+[k/d]\omega-m\omega+k_j\omega,E+i\eta)|\\
    &\qquad \geq \exp(2m(L^d(E,2m)+\langle \log |\det B(\cdot)|\rangle-m^{-\gamma/2}))\cdot \prod_{\ell=1}^{\ell_j}\|\theta+[k/d]\omega-m\omega+k_j\omega-\theta_{j,\ell}\|.
\end{align}
Let $[a_k,b_k]:=[[k/d]d-md+k_jd, [k/d]d-md+k_jd+2md-1]$, and  
\begin{align}
    \tilde{H}_{\theta}|_{[0,dN-1],k}:=H^p_{\theta}|_{[a_k,b_k]} \oplus H_{\theta}|_{[0,dN-1]\setminus [a_k,b_k]}
\end{align}
{{Since $|k_j|<(1-\epsilon)m$, }}
\begin{align}\label{eq:k_to_boundary*}
    \mathrm{dist}(k, \{a_k,b_k\})\geq \epsilon md.
\end{align}
Let $\tilde{G}^{E+i\eta}_{[0,dN-1],k}:=(\tilde{H}_{\theta}|_{[0,dN-1],k}-E)^{-1}$.
By the resolvent identity, we have
\begin{align}\label{eq:G=tG*}
    &G_{[0,dN-1]}^{E+i\eta}(\theta;k,k)\\
    &\qquad=\tilde{G}_{[0,dN-1],k}^{E+i\eta}(\theta;k,k)
   -\left(G^{E+i\eta}_{[0,dN-1]}(\theta)\Gamma_{[a_k,b_k]}\tilde{G}^{E+i\eta}_{[0,dN-1],k}(\theta)\right)(k,k)
\end{align}
Within the product term $\left(G^{E+i\eta}_{[0,dN-1]}(\theta)\Gamma_{[a_k,b_k]}\tilde{G}^{E+i\eta}_{[0,dN-1],k}(\theta)\right)(k,k)$, we bound
\begin{align}\label{eq:G=tG_1*_1*}
    \sup_{x\in \Z} |G^{E+i\eta}_{[0,dN-1]}(\theta;k,x)|\leq \eta^{-1}.
\end{align}
Hence
\begin{align}\label{eq:GGG_1*}
    &\left|\left(G^{E+i\eta}_{[0,dN-1]}(\theta)\Gamma_{[a_k,b_k]}\tilde{G}^{E+i\eta}_{[0,dN-1],k}(\theta)\right)(k,k)\right|\\
    &\qquad \qquad \leq \sum_{\substack{x\in \Z\\ y\in [a_k,b_k]}} \eta^{-1} |\Gamma_{[a_k,b_k]}(x,y)|\cdot |\tilde{G}^{E+i\eta}_{[0,dN-1],k}(\theta;y,k)|.
\end{align}
In order to ensure $\Gamma_{[a_k,b_k]}(x,y)\neq 0$ while $y\in [a_k,b_k]$, it is necessary that \begin{align}
    &y\in [a_k,a_k+d-1]\cup [b_k-d+1,b_k], \text{ and }\\
    &x\in [a_k-d,a_k+d-1] \cup [b_k-d+1,b_k+d].
\end{align}
Thus we can further bound \eqref{eq:GGG_1*} by
\begin{align}\label{eq:GGG_2*}
    &\left|\left(G^{E+i\eta}_{[0,dN-1]}(\theta)\Gamma_{[a_k,b_k]}\tilde{G}^{E+i\eta}_{[0,dN-1],k}(\theta)\right)(k,k)\right|\\
    &\qquad \qquad \leq C_Bd^2\, \eta^{-1} \sup_{y\in [a_k,a_k+d-1]\cup [b_k-d+1,b_k]} |\tilde{G}^{E+i\eta}_{[0,dN-1],k}(\theta;y,k)|,
\end{align}
here $C_B$ is a constant depending only on $B(\cdot)$.
Also, by \eqref{eq:mufn}, we have
\begin{align}\label{eq:G=tG_2*_1*}
    |\tilde{G}^{E+i\eta}_{[0,dN-1],k}(\theta;y,k)|=|\tilde{G}^{E+i\eta}_{[a_k,b_k]}(\theta;y,k)|
    &=\frac{|\mu^{E+i\eta}_{[a_k,b_k],y,k}(\theta)|}{|f_{2m}(\theta+a_k\omega/d,E+i\eta)|}.
\end{align}
Using \eqref{eq:Dtj_lower*} for the lower bound of denominator and Lemma \ref{lem:numerator} for the upper bound of the numerator, we have
\begin{align}\label{eq:G=tG_2*}
    |\tilde{G}^{E+i\eta}_{[a_k,b_k]}(\theta;y,k)|
    &=\frac{|\mu^{E+i\ta}_{[a_k,b_k],y,k}(\theta)|}{|f_{2m}(\theta+a_k\omega/d,E+i\eta)|}\\
    &\qquad \lesssim \tau^{-2\kappa} e^{-\epsilon m(L_d(E)-o(1))}\ll \eta,
\end{align}
in which we used \eqref{eq:k_to_boundary*} and $\|\theta+[k/d]\omega-m\omega+k_j\omega-\theta_{j,\ell}\|\geq\tau$.
Combining \eqref{eq:GGG_2*}, \eqref{eq:G=tG_2*} with \eqref{eq:G=tG*}, we have
\begin{align}\label{eq:G=tG_n*_block'}
    |G^{E+i\eta}_{[0,N-1]}(\theta;k,k)|\leq |\tilde{G}^{E+i\eta}_{[a_k,b_k]}(\theta;k,k)|+C.
\end{align}
Once again, by the lower bound of denominator in \eqref{eq:Dtj_lower*}, and upper bound of numerator in Lemma~\ref{lem:numerator_diag},
\begin{align*}\label{eq:tG_est*_block'}
    |\tilde{G}^{E+i\eta}_{[a_k,b_k]}(\theta;k,k)|
    &=\frac{\mu_{[a_k,b_k],k,k}^{E+i\eta}(\theta)}{|f_{2m}(\theta+a_k\omega/d,E+i\eta)|} 
    \lesssim e^{Cm^{1-\gamma}}\cdot \prod_{\ell=1}^{\ell_j} \|\theta+[k/d]\omega-m\omega+k_j\omega-\theta_{j,\ell}\|^{-1}.
\end{align*}
Thus the overall contribution from Case 2 can be bounded by:
\begin{align}
&\frac{2\eta}{dN}\sum_j\sum_k \chi_{\tilde{I}_j\setminus (\cup_{\ell}(I_{j,\ell}-k_j\omega))}(\theta+[k/d]\omega-m\omega)\cdot \mathrm{Im}(G^{E+i\eta}_{[0,N-1]}(\theta;k,k))\\
&\leq \frac{2\eta}{dN}\sum_j\sum_k \chi_{\tilde{I}_j\setminus (\cup_{\ell}(I_{j,\ell}-k_j\omega))}(\theta+[k/d]\omega-m\omega)\cdot\\
&\qquad\qquad \cdot\left(e^{Cm^{1-\gamma/2}}\cdot \prod_{\ell=1}^{\ell_j}\|\theta+[k/d]\omega-m\omega+k_j\omega-\theta_{j,\ell}\|^{-1}+\eta^{-1}\tau\right)\\
&\lesssim e^{Cm^{1-\gamma/2}}\cdot \eta \int_{\tau}^1 \theta^{-2\kappa^d}\, \mathrm{d}\theta\lesssim \eta^{(1-o(1))} \tau^{1-2\kappa^d},
\end{align}
where we used $1-\gamma/2<\gamma_1$ and hence $\exp(Cm^{1-\gamma/2})<\eta^{-o(1)}$ in the last step. 

\medskip

{\bf Case 3.} If $\theta+[k/d]\omega-m\omega\in I_{j,\ell}$ for some $j,\ell$.

\smallskip

This is the same as Case~3 in the scalar case $d=1$, so we skip it. 

Finally, combining the three cases, we arrive at, as $N\to\infty$,
\begin{align}
    d_{\omega,dN}(\theta,E-\eta,E+\eta)\lesssim m\tau+\eta^{1-o(1)}\tau^{1-2\kappa^d}+\eta^{\frac{1}{2\kappa^d}-o(1)}\lesssim \eta^{\frac{1}{2\kappa^d}-o(1)},
\end{align}
where we used $\tau=\eta^{1/(2\kappa^d)}$ in the last inequality.
This is the claimed H\"older exponent. \qed

\section{Upper bounds for numerators: proof of Lemma \ref{lem:numerator_diag}}
\label{sec:numerator}

We write the monodromy matrices in block form
\begin{align}
M_{n,E'}(\theta) =\left(\begin{array}{c|c}
M_{n,E'}^{UL}(\theta)  & M_{n,E'}^{UR}(\theta) \\
\hline
M_{n,E'}^{LL}(\theta)  & M_{n,E'}^{LR}(\theta)
\end{array}\right),
\end{align}
where each $M_{n,E'}^{\dagger}$ is a $d\times d$ block, $\dagger=UL, UR, LL, LR$.
We will make use of the following recursive relations: for $n=1$,
\begin{align}\label{eq:rec1}
\begin{cases}
M_{1,E'}^{UL}(\theta) =-(V(\theta) -E')B^{-1}(\theta)\\
M_{1,E'}^{UR}(\theta) =-B^{(*)}(\theta)\\
M_{1,E'}^{LL}(\theta) =B^{-1}(\theta)\\
M_{1,E'}^{LR}(\theta) =0
\end{cases}
\end{align}
and for each $n\geq 2$, one has 
\begin{align}\label{eq:rec2}
\begin{cases}
M_{n,E'}^{UL}(\theta)=-M_{n-1,E'}^{UL}(\theta+\omega) (V(\theta)-E')B^{-1}(\theta)+M_{n-1,E'}^{UR}(\theta+\omega) B^{-1}(\theta)\\
M_{n,E'}^{UR}(\theta)=-M_{n-1,E'}^{UL}(\theta+\omega) B^{(*)}(\theta)\\
M_{n,E'}^{LL}(\theta)=-M_{n-1,E'}^{LL}(\theta+\omega) (V(\theta)-E')B^{-1}(\theta)+M_{n-1,E'}^{LR}(\theta+\omega)B^{-1}(\theta)\\
M_{n,E'}^{LR}(\theta)=-M_{n-1,E'}^{LL}(\theta+\omega) B^{(*)}(\theta)
\end{cases}
\end{align}
We now turn to the proof of Lemma~\ref{lem:numerator_diag}.
Recall that within this lemma, $\theta\in \T$, hence $B^{(*)}(\theta)=B^*(\theta)$, $|\det B(\theta)|=|\det B^{(*)}(\theta)|$, and consequently $|\det M_{n,E'}(\theta)|\equiv 1$.
We let 
\begin{align}
R_{x,x}:=\left(\begin{array}{c|c}
P_n(\theta)-E' & {\bf e}_{dn,x}\\
\hline
{\bf e}_{dn,x}^{*} & 0
\end{array}\right),
\end{align}
where ${\bf e}_{m,j}^*=(\delta_j(m-1),...,\delta_j(1),\delta_j(0))$. 
By definition, 
\begin{align}\label{eq:muxy=Rxy}
|\mu_{[0,dn-1],x,x}^{E'}(\theta)|= |\det R_{x,x}|.
\end{align}
For simplicity, we denote $V(\theta+j\omega)-E'=:C_j$, $B(\theta+j\omega)=:B_j$, $B^{(*)}(\theta+j\omega)=:B_j^{(*)}$, and $M_{k,E'}(\theta+j\omega)=:M_k(j)$.
Let $x=y=\ell d+r$ where $r\in [0,d-1]$. Note that by our assumptions, $\ell \in [2,n-3]$.
\begin{align}\label{eq:Rxy}
\qquad
R_{x,x}
=
&\left(\begin{array}{c|c|c|c|c|c|c|c|c|c}
C_{n-1} & B_{n-1}^{(*)} & & &  & & & &B_0 &\\
\hline
B_{n-1} &\ddots &\ddots & & & & & & & \\
\hline
&\ddots &\ddots &\ddots & & & & & &\\
\hline
& &\ddots &\ddots &B_{\ell+1}^{(*)} & & & & &\\
\hline
& & & B_{\ell+1} &C_{\ell} &B_\ell^{(*)} & & & &{\bf e}_{d,r}\\
\hline
& & & &B_\ell &C_{\ell-1} &\ddots & & &\\
\hline
& & & & &\ddots &\ddots &\ddots & &\\
\hline
& & & & & &\ddots &\ddots &B_1^{(*)} &\\
\hline
B_0^{(*)}& & & & & & &B_1 &C_0 &\\
\hline
& & & &{\bf e}_{d,r}^* & & & & &
\end{array}\right)
=:
\left(\begin{matrix}
\mathrm{Row}_1\\
\mathrm{Row}_2\\
\vdots\\
\text{Row}_{n+1}
\end{matrix}\right)
\end{align}

\subsection*{Row operations on \texorpdfstring{$\text{Row}_1$}{Lg}}
Replacing the block-valued $\text{Row}_1$ in 
\eqref{eq:Rxy}, 
$$ \text{Row}_1 \longrightarrow \text{Row}_1-C_{n-1}\cdot B^{-1}_{n-1}\cdot  \text{Row}_2,$$
yields a new row as follows
$$\text{Row}_1^{(1)}=(0, B_{n-1}^{(*)}-C_{n-1} B_{n-1}^{-1}C_{n-2}, -C_{n-1} B^{-1}_{n-1}B^{(*)}_{n-2},0,...,0,B_0,0).$$
Appealing to \eqref{eq:rec1} and \eqref{eq:rec2}, it is easy to see that
$$\text{Row}_1^{(1)}=(0,-M_2^{UL}(n-2)B_{n-2},\, -M_{2}^{UR}(n-2),0,...,0,B_0,0).$$
Replacing 
$$\text{Row}_1^{(1)} \longrightarrow \text{Row}_1^{(1)}+M_{2}^{UL}(n-2) \cdot \text{Row}_3,$$
yields the new row 
$$\text{Row}_1^{(2)}=(0, 0, -M_{2}^{UR}(n-2)+M_{2}^{UL}(n-2) C_{n-3},\, M_{2}^{UL}(n-2)B_{n-3}^{(*)},0,...,0,B_0,0).$$
Appealing to the recursive relations \eqref{eq:rec1} and \eqref{eq:rec2} again, it is easy to check that
$$\text{Row}_1^{(2)}=(0, 0, -M_{3}^{UL}(n-3)B_{n-3},\, -M_{3}^{UR}(n-3),0,...,0,B_0,0).$$
One can iterate this process, and after replacing 
$$\text{Row}_1^{(n-3)}\longrightarrow \text{Row}_1^{(n-3)}+M_{n-2}^{UL}(2)\cdot \text{Row}_{n-1},$$ we arrive at
$$\text{Row}_1^{(n-2)}=(0,...,0,-M_{n-1}^{UL}(1)B_1,\, B_0-M_{n-1}^{UR}(1),M_{n-\ell-1}^{UL}(\ell+1){\bf e}_{d,r}).$$

\subsection*{Row operations on \texorpdfstring{$\text{Row}_2$}{Lg}}
Replacing 
$$\text{Row}_2\longrightarrow \text{Row}_2-C_{n-2}\cdot B_{n-2}^{-1}\cdot \text{Row}_3,$$
yields the following new  row 
\begin{align}
    \text{Row}_2^{(1)}=
    &(B_{n-1}, 0,  B_{n-2}^{(*)}-C_{n-2}B_{n-2}^{-1}C_{n-3}, -C_{n-2}B_{n-2}^{-1}B_{n-3}^{(*)},0,...,0)\\
   =&(B_{n-1}, 0, -M_2^{UL}(n-3)B_{n-3}, -M_2^{UR}(n-3), 0,...,0). 
\end{align}
Repeating this process, and after replacing 
$$\text{Row}_2^{(n-4)}\longrightarrow \text{Row}_2^{(n-4)}+M_{n-3}^{UL}(2)\cdot  \text{Row}_{n-1},$$
we arrive at
\begin{align}
    \text{Row}_2^{(n-3)}=(B_{n-1},0,...,0,-M_{n-2}^{UL}(1)B_1,-M_{n-2}^{UR}(1),M^{UL}_{n-\ell-2}(\ell+1){\bf e}_{d,r}).
\end{align}
Next, we perform row reduction on $\text{Row}_{n-\ell+1}$. Recall that $3\leq \ell\leq n-2$.

\subsection*{Row operations on \texorpdfstring{$\text{Row}_{n-\ell+1}$ (the row containing $C_{\ell-1}$)}{Lg}.}

Replacing 
$$\text{Row}_{n-\ell+1}\longrightarrow \text{Row}_{n-\ell+1}-C_{\ell-1} B_{\ell-1}^{-1}\cdot \text{Row}_{n-\ell+2},$$
yields a new row as follows
\begin{align}
\text{Row}_{n-\ell+1}^{(1)}
=&(0,...,0,B_{\ell},0,B_{\ell-1}^{(*)}-C_{\ell-1} B_{\ell-1}^{-1}C_{\ell-2}, -C_{\ell-1}B_{\ell-1}^{-1}B_{\ell-2}^*,0,...,0)\\
=&(0,...,0,B_{\ell},0,-M_{2}^{UL}(\ell-2)B_{\ell-2},\, -M_{2}^{UR}(\ell-2),0,...,0).
\end{align}
where the matrix $B_{\ell}$ is the $(n-\ell)-$th entry (from the left) of the block-valued vector above, and we have used \eqref{eq:rec1} and \eqref{eq:rec2} to obtain the second line.
Repeating this process, and after replacing
$$\text{Row}_{n-\ell+1}^{(\ell-3)}\longrightarrow \text{Row}_{n-\ell+1}^{(\ell-3)}+M_{\ell-2}^{UL}(2) \cdot \text{Row}_{n-1},$$ 
we arrive at
$$\text{Row}_{n-\ell+1}^{(\ell-2)}=(0,...,0,B_{\ell},0,...,0,-M_{\ell-1}^{UL}(1)B_1,-M_{\ell-1}^{UR}(1),0).$$
At this point, $R_{x,x}$ becomes $R^{(1)}_{x,x}
=$
\begin{align}\label{eq:R2xy_big_matrix}
&\!\!\!\!\!\!\!\!\!\!\!\!\!\!\!\!\!\!\!\!\!\!\!\!\!\!\!\left(\begin{array}{c|c|c|c|c|c|c|c|c|c|c|c|c}
0 & 0 &0 &\cdots & &  & & &\cdots &0 &-M_{n-1}^{UL}(1)B_1 &B_0-M_{n-1}^{UR}(1) & M_{n-\ell-1}^{UL}(\ell+1){\bf e}_{d,r} \\
\hline
B_{n-1} &0 &0 &\cdots & & & & &\cdots &0 &-M_{n-2}^{UL}(1)B_1 &-M_{n-2}^{UR}(1) &M^{UL}_{n-\ell-2}(\ell+1){\bf e}_{d,r} \\
\hline 
0&B_{n-2} &C_{n-3} &B_{n-3}^{(*)} & & & & & & & & &\\
\hline
\vdots& &\ddots &\ddots &\ddots & & & & & & & &\\
\hline
& & & \ddots &\ddots &\ddots & & & & & & &\\
\hline
\vdots & & & &B_{\ell+1} &C_{\ell} &B_{\ell}^{(*)} & & & & & &{\bf e}_{d,r}\\
\hline
0&\cdots & &\cdots &0 &B_{\ell} &0 &0 &\cdots &0 &-M_{\ell-1}^{UL}(1)B_1 &-M_{\ell-1}^{UR}(1) &0\\
\hline
\vdots & & & & &0 &B_{\ell-1} &C_{\ell-2} &B_{\ell-2}^{(*)} & & &0 &\vdots\\
\hline
& & & & & & &\ddots &\ddots &\ddots & &\vdots & \\
\hline
\vdots& & & & & & & &\ddots &\ddots &\ddots &0 &\vdots\\
\hline
0& & & & & & & & &\ddots &\ddots &B_1^{(*)} &0\\
\hline
B_0^{(*)} &0 &\cdots & & & & & &\cdots &0 &B_1 &C_0 &0\\
\hline
0 &\cdots & &\cdots & 0 &{\bf e}_{d,r}^* & 0 &\cdots & &\cdots &  0 & 0 & 0 
\end{array}\right)
\end{align}
The first, second, $(n-\ell+1)$-th, $n$-th, and $(n+1)$-th rows are
\begin{align}
\left(\begin{matrix}
\text{Row}_1^{(n-2)}\\
\text{Row}_2^{(n-3)}\\
\text{Row}_{n-\ell+1}^{(\ell-2)}\\
\text{Row}_{n}\\
\text{Row}_{n+1}
\end{matrix}
\right)
=\left(\begin{array}{ccccccccccc}
0 &0 &\cdots &0 &0 &0 &\cdots &0 &-M_{n-1}^{UL}(1)B_1 &B_0-M_{n-1}^{UR}(1) &M_{n-\ell-1}^{UL}(\ell+1){\bf e}_{d,r}\\
B_{n-1} &0 &\cdots &0 &0 &0 &\cdots &0 &-M_{n-2}^{UL}(1)B_1 &-M_{n-2}^{UR}(1) &M^{UL}_{n-\ell-2}(\ell+1){\bf e}_{d,r}\\
0 &0 &\cdots &0 &B_{\ell} &0 &\cdots &0 &-M_{\ell-1}^{UL}(1)B_1 &-M_{\ell-1}^{UR}(1) &0\\
B_0^{(*)} &0 &\cdots &0 &0 &0 &\cdots &0 &B_1 &C_0 &0\\
0 &0 &\cdots &0 &{\bf e}_{d,r}^* &0 &\cdots &0 &0 &0 &0
\end{array}\right),
\end{align}
in which only columns $1$, $n-\ell$, $n-1$, $n$, $n+1$ are non-vanishing.
Let
\begin{align}
    S_1:=\left(\begin{array}{ccccc}
0 &0 &-M_{n-1}^{UL}(1)B_1 &B_0-M_{n-1}^{UR}(1) &M_{n-\ell-1}^{UL}(\ell+1){\bf e}_{d,r}\\
B_{n-1} &0 &-M_{n-2}^{UL}(1)B_1 &-M_{n-2}^{UR}(1) &M^{UL}_{n-\ell-2}(\ell+1){\bf e}_{d,r}\\
0 &B_{\ell} &-M_{\ell-1}^{UL}(1)B_1 &-M_{\ell-1}^{UR}(1) &0\\
B_0^{(*)} &0 &B_1 &C_0 &0\\
0 &{\bf e}^*_{d,r} & 0 &0 &0
    \end{array}\right)
\end{align}
be the non-vanishing $(4d+1)\times (4d+1)$ submatrix of rows $1$, $2$, $n-\ell+1$, $n$ and $n+1$.
Let
\begin{align}
    S_2:=
    \left(\begin{array}{cccccccccccccccc}
       B_{n-2} &C_{n-3} &B_{n-3}^{(*)} & & & & & & & & & & & & &\\
          &B_{n-3}       &C_{n-4} &\ddots & & & & & & & & & & & &\\
          &        &\ddots  &\ddots & & & & & & & & & & & &\\
          & & & & &\ddots &C_{\ell+2} &B_{\ell+2}^{(*)} & & & & & & & &\\
          & & & & &  &B_{\ell+2} &C_{\ell+1} &0 & & & & & & &\\
          & & & & &  & &B_{\ell+1} &B_{\ell}^{(*)} &0 & & & & & &\\
          & & & & &  & & &B_{\ell-1} &C_{\ell-2} &B_{\ell-2}^{(*)} & & & & &\\
          & & & & &  & & & &B_{\ell-2} &\ddots &\ddots & & & &\\
          & & & & &  & & & & &\ddots & & & & &\\
          & & & & &  & & & & & & & &B_4 &C_3 &B_3^{(*)}\\
          & & & & &  & & & & & & & & &B_3 &C_2   \\ 
          & & & & &  & & & & & & & & & &B_2          
    \end{array}\right),
\end{align}
which is the submatrix of $R^{(1)}_{x,x}$ after deleting rows $1$, $2$, $n-\ell+1$, $n$, $n+1$ and columns $1$, $n-\ell$, $n-1$, $n$, $n+1$.
Hence 
\begin{align}\label{eq:num1}
|\det R^{(1)}_{x,x}|=
|\det S_1|\cdot |\det S_2|=\prod_{\substack{j=2\\ j\neq \ell}}^{n-2}|\det B_j|\cdot |\det S_1|.
\end{align}
It suffices to compute $\det S_1$.
Expanding determinants and using that $M_{k-2}^{UL}(j)=B_{k+j-2}M_{k-1}^{LL}(j)$, $M_{k-2}^{UR}(j)=B_{k+j-2}M_{k-1}^{LR}(j)$, we have
\begin{align}\notag
|\det S_1|=
&\left|\det\left(\begin{array}{ccccc}
0 &0 &-M_{n-1}^{UL}(1)B_1 &B_0-M_{n-1}^{UR}(1) &M_{n-\ell-1}^{UL}(\ell+1){\bf e}_{d,r}\\
B_{n-1} &0 &-B_{n-1}M_{n-1}^{LL}(1)B_1 &-B_{n-1}M_{n-1}^{LR}(1) &B_{n-1}M^{LL}_{n-\ell-1}(\ell+1){\bf e}_{d,r}\\
0 &B_{\ell} &-M_{\ell-1}^{UL}(1)B_1 &-M_{\ell-1}^{UR}(1) &0\\
B_0^{(*)} &0 &B_1 &C_0 &0\\
0 &{\bf e}^*_{d,r} & 0 &0 &0
    \end{array}\right)\right| \notag\\
=
&|\det B_1|\cdot \left|\det\left(\begin{array}{ccccc}
0 &0 &-M_{n-1}^{UL}(1) &B_0-M_{n-1}^{UR}(1) &M^{UL}_{n-\ell-1}(\ell+1){\bf e}_{d,r}\\
B_{n-1} &0 &-B_{n-1}M_{n-1}^{LL}(1) &-B_{n-1}M_{n-1}^{LR}(1) &B_{n-1}M^{LL}_{n-\ell-1}(\ell+1){\bf e}_{d,r}\\
0 &B_{\ell} &-M_{\ell-1}^{UL}(1) &-M_{\ell-1}^{UR}(1) &0\\
B_0^{(*)} &0 &I_d &C_0 &0\\
0 &{\bf e}^*_{d,r} & 0 &0 &0
    \end{array}\right)\right|\notag
    \end{align}
    Factoring out $B_{n-1}$ and $B_0^{(*)}$ in the first column from corresponding rows yields 
    \begin{align*}
        |\det S_1|\\
=&\prod_{j=0,1,n-1}|\det B_j|\cdot
\left|\det\left(\begin{array}{ccccc}
0 &0 &-M_{n-1}^{UL}(1) &B_0-M_{n-1}^{UR}(1) &M^{UL}_{n-\ell-1}(\ell+1){\bf e}_{d,r}\\
I_d &0 &-M_{n-1}^{LL}(1) &-M_{n-1}^{LR}(1) &M^{LL}_{n-\ell-1}(\ell+1){\bf e}_{d,r}\\
0 &B_{\ell} &-M_{\ell-1}^{UL}(1) &-M_{\ell-1}^{UR}(1) &0\\
I_d &0 &(B_0^{(*)})^{-1} &(B_0^{(*)})^{-1} C_0 &0\\
0 &{\bf e}^*_{d,r} & 0 &0 &0
    \end{array}\right)\right|\notag\\
    \end{align*}
    which furthermore equals
    \begin{align}\label{eq:num2}
&\prod_{j=0,1,n-1}|\det B_j| \cdot 
\left|\det\left(\begin{array}{cccccc}
0& 0 &0 &-M_{n-1}^{UL}(1) &B_0-M_{n-1}^{UR}(1) &M^{UL}_{n-\ell-1}(\ell+1){\bf e}_{d,r}\\
0& I_d &0 &-M_{n-1}^{LL}(1) &-M_{n-1}^{LR}(1) &M^{LL}_{n-\ell-1}(\ell+1){\bf e}_{d,r}\\
0 &0 &B_{\ell} &-M_{\ell-1}^{UL}(1) &-M_{\ell-1}^{UR}(1) &0\\
I_d &0 &0 &0 &-B_0 &0\\
0& I_d &0 &(B_0^{(*)})^{-1} &(B_0^{(*)})^{-1} C_0 &0\\
0& 0 &{\bf e}^*_{d,r} & 0 &0 &0
    \end{array}\right)\right|\notag\\
=&\prod_{j=0,1,n-1}|\det B_j|\cdot 
\left|\det\left(\begin{array}{c|c|c|c}
\left(\begin{matrix}0& 0\\
0 &I_d\end{matrix}\right)& &\left(\begin{matrix}0 & B_0\\ 0 & 0\end{matrix}\right)-M_{n-1}(1) &M_{n-\ell-1}(\ell+1)\left(\begin{matrix} I_d \\ 0\end{matrix}\right) {\bf e}_{d,r}\\
\hline
 & B_{\ell} &-M_{\ell-1}^U(1) &  \\
 \hline
I_{2d} & &-M_1^{-1}(0) &\\
\hline
 &{\bf e}^*_{d,r} & &
    \end{array}\right)\right|\notag\\
=:&\prod_{j=0,1,n-1}|\det B_j|\cdot  |\det S_2|,
\end{align}
in which $M_{\ell-1}^U(1)=(M_{\ell-1}^{UL}(1), M_{\ell-1}^{UR}(1))=(I_d, 0)\cdot M_{\ell-1}(1)$. Empty spaces are to be filled in by zeros.  
We carry out the following row operations on $S_2$ that do not alter the determinant of $S_2$. Below $\text{row}_j$ refers to the $j$-th (block-valued) row of $S_2$. Replacing 
$$\text{row}_1\longrightarrow \text{row}_1-\left(\begin{matrix} 0 &0\\ 0 &I_d\end{matrix}\right) \cdot \text{row}_3,$$
and
$$\text{row}_4\longrightarrow \text{row}_4-{\bf e}_{d,r}^* B_{\ell}^{-1}\cdot \text{row}_2,$$
we arrive at
\begin{align}
S_2^{(1)}=&\left(\begin{array}{c|c|c|c}
& &\left(\begin{matrix}0 & B_0\\ 0 & 0\end{matrix}\right)-M_{n-1}(1)+\left(\begin{matrix}0 &0 \\ 0 &I_d\end{matrix}\right)M_1^{-1}(0) &M_{n-\ell-1}(\ell+1)\left(\begin{matrix} I_d \\ 0\end{matrix}\right) {\bf e}_{d,r}\\
\hline
 & B_{\ell} &-(I_d, 0) \cdot M_{\ell-1}(1) &  \\
 \hline
I_{2d} & &-M_1^{-1}(0) & \\
\hline
 & &{\bf e}_{d,r}^*B_{\ell}^{-1}(I_d, 0)\cdot M_{\ell-1}(1) &
    \end{array}\right)\\
    =&\left(\begin{array}{c|c|c|c}
& &\left(\begin{matrix}0 & B_0\\ 0 & 0\end{matrix}\right)M_1(0)M_1^{-1}(0)-M_{n}(0)M_1^{-1}(0)+\left(\begin{matrix}0 &0 \\ 0 &I_d\end{matrix}\right)M_1^{-1}(0)&M_{n-\ell-1}(\ell+1)\left(\begin{matrix} I_d \\ 0\end{matrix}\right) {\bf e}_{d,r}\\
\hline
 & B_{\ell} &-(I_d, 0)\cdot M_{\ell}(0)\, M_1^{-1}(0) &  \\
 \hline
I_{2d} & &-M_1^{-1}(0) &\\
\hline
 & &{\bf e}_{d,r}^*B_{\ell}^{-1}(I_d, 0)\cdot M_{\ell}(0)\, M_1^{-1}(0) &
    \end{array}\right).
\end{align}
Expanding the determinant along the first two columns, and noting that $\left(\begin{matrix} 0 &B_0\\ 0 &0\end{matrix}\right) M_1(0)=\left(\begin{matrix}I_d &0\\ 0 &0\end{matrix}\right)$, yields
\begin{align}\label{eq:num3}
|\det S_2^{(1)}|&=\frac{|\det B_{\ell}|}{|\det M_1(0)|}\cdot  
\left|\det\left(\begin{array}{c|c}
    -M_n(0)+I_{2d} &M_{n-\ell-1}(\ell+1)\left(\begin{matrix} I_d\\ 0\end{matrix}\right) {\bf e}_{d,r}\\
    \hline
    {\bf e}_{d,r}^*B_{\ell}^{-1}(I_d, 0)\cdot M_{\ell}(0) &0
\end{array}\right)\right| \notag\\
&=:|\det B_{\ell}| \cdot |\det S_2^{(2)}|.
\end{align}
Here after factoring out $M_{\ell}(0)$ and $M_{n-\ell-1}(\ell+1)$ in $S_2^{(2)}$, we have
\begin{align}
    |\det S_2^{(2)}|
    =&|\det M_{n-\ell-1}(\ell+1)|\cdot |\det M_{\ell}(0)|\cdot 
    \left|\det \left(\begin{array}{c|c}
    -M_1(\ell)+M_{n-\ell-1}^{-1}(\ell+1)M_{\ell}^{-1}(0) & \left(\begin{matrix} I_d\\ 0\end{matrix}\right) {\bf e}_{d,r}\\ 
    \hline 
    {\bf e}_{d,r}^* B_{\ell}^{-1}(I_d, 0) &0\end{array}\right)\right|\\
    \leq & \|B_{\ell}^{-1}\|\cdot  \|{\bigwedge}^{2d-1} (-M_1(\ell)+M_{n-\ell-1}^{-1}(\ell+1)M_{\ell}^{-1}(0))\|.
\end{align}
Now in order to compute the last wedge product, we write
\begin{align}\label{eq:wedge_2d-1}
    -M_1(\ell)+M_{n-\ell-1}^{-1}(\ell+1)M_{\ell}^{-1}(0)
    =&
    -M_1(\ell)+\Omega (M_{n-\ell-1}(\ell+1))^* (M_{\ell}(0))^* \Omega \notag\\
    =&\Omega (-\Omega M_1(\ell)\Omega+(M_{n-\ell-1}(\ell+1))^* (M_{\ell}(0))^*)\Omega.
\end{align}
Then the claimed result will following from the following general lemma.
\begin{lemma}\label{lem:wedge_upper}
Let $W,A,B$ be $2d\times 2d$ matrices. Then for any $\ell\in [1,2d]$, the following holds
\begin{align}
    \|{\bigwedge}^{\ell} (W+AB)\|\leq C_d(\|W\|+1)^{\ell} \max_{j=1}^{\ell} (\|{\bigwedge}^j A\|\cdot \|{\bigwedge}^j B\|),
\end{align}
for some constant $C_d$ depending only on $d$.
\end{lemma}
\begin{proof}
    By the singular value decomposition,
    \begin{align}
        W+AB=W+U_A \Sigma_A V_A^* U_B \Sigma_B V_B^*=U_A(U_A^*WV_B+\Sigma_A V_A^* U_B \Sigma_B)V_B^*.
    \end{align}
    Denote $U_A^*WV_B=:\tilde{W}$ and $V_A^*U_B=:F$. Clearly $\|\tilde{W}\|=\|W\|$ and $\|F\|=1$. 
  Furthermore, 
    \begin{align}
        \|{\bigwedge}^{\ell}(\tilde{W}+\Sigma_A V_A^*U_B \Sigma_B)\|
        &=\|{\bigwedge}^{\ell} (\tilde{W}_{m,k}+ \sigma_m(A) F_{m,k} \sigma_k(B))_{1\leq m,k\leq 2d}\|\\
        &\leq C_d\sup_{(m_1,...,m_{\ell}), (k_1,...,k_{\ell})} \prod_{s,t=1}^{\ell} |\tilde{W}_{m_s,k_t}+\sigma_{m_s}(A)F_{m_s,k_t}\sigma_{k_t}(B)|\\
        &\leq C_d \prod_{j=1}^{\ell} (\|\tilde{W}\|+\sigma_j(A)\sigma_j(B))\\
        &\leq C_d(\|W\|+1)^{\ell} \max_{j=1}^{\ell} (\|{\bigwedge}^j A\|\cdot \|{\bigwedge}^j B\|),
    \end{align}
    which implies the desired inequality.
\end{proof}
Applying Lemmas \ref{lem:wedge_upper} and  to \eqref{eq:wedge_2d-1} we conclude that
\begin{align}
    &\|{\bigwedge}^{2d-1} (-M_{1,E'}(\ell)+M_{n-\ell-1,E'}^{-1}(\ell+1)M_{\ell,E'}^{-1}(0))\|\\
    &\leq C_d (\|M_{1,E'}(\cdot)\|_{L^\infty(\T)}+1)^{2d-1}\max_{j=1}^{2d-1} \|{\bigwedge}^j M_{n-\ell-1,E'}(\ell+1)\| \cdot \|{\bigwedge}^j M_{\ell,E'}(0)\|
\end{align}
{{For $j=d$, by Lemma \ref{lem:upperbd_E'}, since $|E'-E|<n^{-2}e^{-n^{\gamma}}$,
\begin{align}\label{eq:w_j=d}
&\|{\bigwedge}^d M_{n-\ell-1,E'}(\ell+1)\| \cdot \|{\bigwedge}^d M_{\ell,E'}(0)\|\\
&\leq C^2\exp((n-\ell-1)L^d(E)+2(n-\ell-1)^{1-\gamma}) \cdot \exp(\ell L^d(E)+2\ell^{1-\gamma})\\
&\leq C^2\exp(nL^d(E)+4n^{1-\gamma}).
\end{align}
For $j\neq d$, by Lemma \ref{lem:upper_j_neq_d}, 
\begin{align}\label{eq:w_jneqd}
&\|{\bigwedge}^j M_{n-\ell-1,E'}(\ell+1)\| \cdot \|{\bigwedge}^j M_{\ell,E'}(0)\|\\
&\leq C^2\exp((n-\ell-1)(L^j(E)+o(1))) \cdot \exp(\ell (L^j(E)+o(1)))\\
&\leq C^2\exp(n(L^j(E)+o(1)))\\
&\leq C^2\exp(n L^d(E)).
\end{align}
where we used $L^d(E)\geq \max_{j\neq d} L^j(E)+L_d(E)$ in the last inequality.}}
Inserting the estimates \eqref{eq:w_j=d} and \eqref{eq:w_jneqd} back into $S_2^{(2)}$ and tracing the bounds all the way back to $R_{x,x}$, we conclude that
\begin{align}
    |R_{x,x}|&\leq C_{d,B,V}\|B^{-1}\|_{L^\infty(\T)}\left(\prod_{j=0}^{n-1} |\det B_j|\right) \exp(n(L^d(E)+4n^{1-\gamma}))\\
    &\leq 
    \exp(n(L^d(E)+\langle |\det B(\cdot)|\rangle+5n^{1-\gamma})),
\end{align}
for $n$ large enough, as claimed.
\qed

\end{document}